\documentclass[12pt, reqno, a4paper]{amsart}
\usepackage{amsmath,mathrsfs}
\usepackage{amssymb}
\usepackage{color}
\usepackage{calligra}
\usepackage{dsfont}
\usepackage{pstricks,pst-node}
\usepackage[latin1]{inputenc}
\usepackage{todonotes}

\newcommand\NoBlackBoxes{\global\overfullrule0pt}
\NoBlackBoxes

\usepackage{enumitem}
\setitemize{noitemsep,topsep=0pt,parsep=0pt,partopsep=0pt}

\newcommand{\eps}{\varepsilon}

\newcommand{\N}{\mathbb{N}}

\newcommand{\Z}{\mathbb{Z}}

\renewcommand{\P}{\mathbb{P}}

\newcommand{\Cov}{\mathop{\mathrm{Cov}}\nolimits}

\newcommand{\eee}{{\rm e}}
\newcommand{\dd}{{\rm d}}

\makeatletter
\let\serieslogo@\relax
\let\@setcopyright\relax

\theoremstyle{plain}
\newtheorem{theorem}{Theorem}[section]
\newtheorem{lemma}[theorem]{Lemma}
\newtheorem{corollary}[theorem]{Corollary}
\newtheorem{proposition}[theorem]{Proposition}

\theoremstyle{definition}

\theoremstyle{remark}

\renewcommand{\P}{{\mathbb{P}}}
\newcommand{\E}{{\mathbb{E}}}
\newcommand{\R}{{\mathbb{R}}}

\newcommand{\C}{\mathbb{C}}

\newcommand{\V}{\mathbb{V}}

\renewcommand{\epsilon}{\varepsilon}
\renewcommand{\phi}{\varphi}

\numberwithin{equation}{section}

\begin{document}

\setcounter{page}{1}

\title[Fluctuations for Ising models on Erd\H{o}s-R\'enyi graphs]{Fluctuations of the Magnetization for Ising models on Erd\H{o}s-R\'enyi random graphs -- the regimes of small $\boldsymbol{p}$ and the critical temperature}

\author[Zakhar Kabluchko]{Zakhar Kabluchko}
\address[Zakhar Kabluchko]{Fachbereich Mathematik und Informatik,
Universit\"at M\"unster,
Einsteinstra\ss e 62,
48149 M\"unster,
Germany}

\email[Zakhar Kabluchko]{zakhar.kabluchko@uni-muenster.de}

\author[Matthias L\"owe]{Matthias L\"owe}
\address[Matthias L\"owe]{Fachbereich Mathematik und Informatik,
Universit\"at M\"unster,
Einsteinstra\ss e 62,
48149 M\"unster,
Germany}

\email[Matthias L\"owe]{maloewe@math.uni-muenster.de}

\author[Kristina Schubert]{Kristina Schubert}
\address[Kristina Schubert]{ Fakult\"at f\"ur Mathematik, TU Dortmund, Vogelpothsweg 87, 44227 Dortmund,
Germany}

\email[Kristina Schubert]{kristina.schubert@tu-dortmund.de}


\date{\today}



\newcommand{\wlim}{\mathop{\hbox{\rm w-lim}}}
\newcommand{\na}{{\mathbb N}}
\newcommand{\re}{{\mathbb R}}

\newcommand{\vep}{\varepsilon}

\begin{abstract}
We continue our analysis of Ising models on the (directed) Erd\H{o}s-R\'enyi random graph. This graph is constructed on $N$ vertices and every edge has probability $p$ to be present.
These models were introduced by Bovier and Gayrard [\emph{J.\ Stat.\ Phys.}, 1993] and analyzed by the authors in a previous note, in which we consider the case of $p=p(N)$ satisfying $p^3N^2\to +\infty$ and $\beta <1$. 
In the current note we prove a quenched Central Limit Theorem for the magnetization  for $p$ satisfying $pN \to \infty$ in the high-temperature regime $\beta<1$. 

We also show a non-standard Central Limit Theorem for $p^4N^3 \to \infty$ at the critical temperature $\beta=1$. For $p^4N^3 \to 0$ we obtain a Gaussian limiting distribution for the magnetization. Finally, on the critical line $p^4N^3 \to c$ the limiting distribution for the magnetization contains a quadratic component as well as a $x^4$-term. Hence, at $\beta=1$ we observe a phase transition in $p$ for the fluctuations of the magnetization. 
\end{abstract}

\maketitle

\section{Introduction and main results}
\subsection{Description of the model}
In this paper we continue our investigation of Ising models on the Erd\H{o}s-R\'enyi random graph. They fall into the category of disordered ferromagnets, see e.g.~\cite{froehlichlecture} for a classic survey and \cite{georgii_dilute} for first mathematically rigorous results. The model we are investigating in the present note was introduced by Bovier and Gayrard in \cite{BG93b}. In the same article the authors also prove a law of large number type result, which we will describe later.

The general ``architecture'' of the model is that of a realization of a directed Erd\H{o}s-R\'enyi graph $G=G(N,p)$. This means that on the vertex set $\{1, \ldots, N\}$ the directed edge $(i,j)$ is realized with probability $p\in (0,1]$. Note that the case $i=j$ is allowed, so the graph $G$ may have loops. The indicator random variables $\vep_{i,j}, i,j \in \{1, \ldots, N\}$, which indicate whether an edge $(i,j)$ is present or not, are assumed to be independent. By definition, $(\vep_{i,j})_{i,j=1}^{N}$ are thus independent random variables with
$$
\P[\vep_{i,j} = 1] = p, \qquad \P[\vep_{i,j} = 0] = 1- p.
$$
One major difference of this note from our previous article \cite{KaLS19} is that here we only assume that $p=p(N)$ and $N$ satisfy $pN \to \infty$ as $N\to\infty$. Note that this even allows for $p$ smaller than $\log(N)/N$, which is the threshold for $G$ being connected asymptotically almost surely. This is to be contrasted to the situation in \cite{KaLS19}, where we had to assume that $p^3N^2 \to \infty$. Another important difference is that we are also able to prove a (non-standard) Central Limit Theorem at the critical inverse temperature $\beta=1$ in this note.

Returning to the definition of our model, the Hamiltonian or energy function of the Ising model on $G$ (i.e.~a fixed realization of the Erd\H{o}s-R\'enyi random graph) is a function $H:= H_N: \{-1,+1\}^N \to\R$. This function is given by
\begin{equation}\label{hamilCW}
H(\sigma) = - \frac 1 {2Np} \sum_{i,j=1}^N \vep_{i,j} \sigma_i \sigma_j
\end{equation}
for $\sigma = (\sigma_1,\ldots,\sigma_N) \in \{-1,+1\}^N$.
With such an energy function $H$ we may associate a Gibbs measure on $\{-1,+1\}^N$. This is a random probability measure with respect to the randomness encoded by the $(\vep_{i,j})_{i,j=1}^N$. It is given by
\begin{equation}\label{gibbs}
\mu_\beta (\sigma):=\frac 1 {Z_{N}(\beta)} \exp(-\beta H(\sigma)), \qquad \sigma \in \{-1,+1\}^N,
\end{equation}
where $\beta \ge 0$ is called the inverse temperature. The normalizing constant is given by
\begin{equation}\label{partition}
Z_N(\beta):= \sum_{\sigma \in \{-1,+1\}^N} \exp(-\beta H(\sigma)).
\end{equation}
The well known Curie-Weiss model is the model with $p\equiv 1$ (of course then $H$ and $\mu_\beta$ are no longer random). It has been intensively studied in the past, see  \cite{Ellis-EntropyLargeDeviationsAndStatisticalMechanics}, for a survey. One of the first findings is that the Curie-Weiss model undergoes a phase transition at $\beta=1$. This can be seen, among others, when considering the magnetization per particle
\begin{equation}\label{eq:magnet}
m_N(\sigma)= \frac{\sum_{i=1}^N \sigma_i}{N}= \frac{|\sigma|}{N}.
\end{equation}
Here we have set
\begin{equation}\label{eq:sigma}
|\sigma|:= N m_N(\sigma)= \sum_{i=1}^N \sigma_i.
\end{equation}
In the standard Curie-Weiss model the law of $m_N$ under the Gibbs measure converges to
$$
\frac 12 (\delta_{m^+(\beta)}+\delta_{m^-(\beta)}),
$$
where $\delta_x$ is the Dirac-measure in a point $x$, $m^+(\beta)$ is the largest solution of
\begin{equation}\label{eq:CWeq}
z= \tanh(\beta z),
\end{equation}
and $m^-(\beta)=-m^+(\beta)$.
Since for $\beta\le 1$ the equation \eqref{eq:CWeq} has only the solution $z=0$, in this so-called high temperature regime $m_N$ converges to $0$ in probability. For $\beta>1$ the largest solution of \eqref{eq:CWeq} is strictly positive. Hence the magnetization $m_N$ is asymptotically concentrated in two values, a positive one and a negative one.
As shown by Bovier and Gayrard in \cite{BG93b}, the same holds true for the dilute Curie-Weiss Ising model defined by \eqref{hamilCW} and \eqref{gibbs} if $p N \to \infty$. For $\beta \le 1$ the magnetization $m_N$ converges to $0$ in probability under the Gibbs measure for almost all realizations of the random graph. For $\beta >1$ the distribution of the magnetization again converges to $\frac 12 (\delta_{m^+(\beta)}+\delta_{m^-(\beta)})$, where $m^+(\beta)$ and $m^-(\beta)$ are defined as above.

Indeed, there is also a Central Limit Theorem for  the  magnetization in the Curie-Weiss model
(see, e.g.~\cite{Ellis_Newman_78b}, \cite{Ellis-EntropyLargeDeviationsAndStatisticalMechanics}, \cite{EL10}, \cite{Chatterjee_Shao}) when $\beta <1$. In this case $\sqrt N m_N$ converges in distribution to a normal random variable with mean $0$ and variance $\frac 1 {1-\beta}$. Moreover, at $\beta=1$, there is no such standard Central Limit Theorem and one has to scale in a different way. The result is that $\sqrt[4] N m_N$ converges in distribution to a non-normal random variable with density proportional to $\exp(-\frac 1 {12} x^4)$ with respect to the Lebesegue measure.

Inspired by the first of these results, in \cite{KaLS19} we showed the following:
Denote by $L_N$ the following {\it random} element of the space of probability measures on $\R$, denoted by $\mathcal M(\R)$:
\begin{equation}\label{eq:def_L_N}
L_N := \frac 1{Z_N(\beta)} \sum_{\sigma\in \{-1,+1\}^N} \eee^{-\beta H(\sigma)}  \delta_{\frac 1 {\sqrt N} \sum_{i=1}^N \sigma_i}.
\end{equation}
(Note that $L_N$ is random, because it depends on the random variables $\vep_{i,j}, i,j \in \{1, \ldots, N\}$.)
Then, if $p^3 N^2 \to \infty$, we showed that $L_N$ converges in probability to the normal distribution with mean 0 and variance $\frac{1}{1-\beta}$, denoted by $\mathfrak N_{0, 1/(1-\beta)}$. Here, we consider $L_N$ as a random variable with values in the space $\mathcal M(\R)$ endowed with some metric generating the weak topology, and $\mathfrak N_{0, 1/(1-\beta)}$ as a deterministic element of the same space.

Note that the situation we analyzed in \cite{KaLS19} is essentially different from the situation, when the topology of the graph is locally tree like, as is the case e.g.\  for sparse Erd\H{o}s-R\'enyi graphs or in some other models of random graphs. The corresponding situation was analyzed, e.g.\ in \cite{Dembo_Montanari_2010a} and \cite{Dembo_Montanari_2010b} as well as  in \cite{van_der_Hofstad_et_al_2010},
\cite{van_der_Hofstad_et_al_2014}, \cite{van_der_Hofstad_et_al_2015}, \cite{van_der_Hofstad_et_al_2015b}, and \cite{van_der_Hofstad_et_al_2015c}.

Comparing our results in \cite{KaLS19} with the fluctuation results for the Curie-Weiss model cited above, immediately raises two questions: First, is the restriction $N^2 p^3~\to~\infty$ necessary for the above statement to hold? And second, can we say anything about the fluctuations of the magnetization at $\beta=1$? For both these questions our techniques in \cite{KaLS19} were insufficient. In a nutshell, the key idea there was to consider
the following generalization of the partition function:
\begin{equation}\label{ZN(g)}
Z_N(\beta, g):= \sum_{\sigma \in \{-1, +1\}^N} \eee^{-\beta H(\sigma)} g\left( \frac{\sum_{i=1}^N \sigma_i}{\sqrt N} \right).
\end{equation}
Here we took $g \in \mathcal{C}_b(\R)$ to be a continuous, bounded function on $\R$. Note that
$Z_N(\beta, g)$ is a generalization of the partition function since for $g \equiv 1$ we obtain $Z_N(\beta)= Z_N(\beta, 1)$, i.e.\  the partition function defined in~\eqref{partition}.
Moreover, we see that
\begin{equation}
\E_{\mu_\beta }\left[ g\left( \frac{\sum_{i=1}^N \sigma_i}{\sqrt N} \right)\right]= \frac{Z_N(\beta, g)}{Z_N(\beta)},
\end{equation}
where, for a fixed disorder $(\eps_{i,j})_{i,j=1}^N$,  $\E_{\mu_\beta}$ denotes the expectation with respect to the Gibbs measure $\mu_\beta$.

In \cite{KaLS19} we were able to show that for $g\ge 0$, $g \not\equiv 0$, $\beta<1$ and $N^2p^3\to \infty$ the variance of $Z_N(\beta, g)$, i.e.\, $\V(Z_N(\beta, g)$, is $o(\E(Z_N(\beta, g))^2)$. Since the set
$\mathcal{C}_b(\R) \cap \{g: g\ge 0, g \not\equiv 0\}$ is convergence determining for weak convergence, this was enough to prove our main result. However, this approach essentially needs the conditions $\beta<1$ and $N^2p^3\to \infty$.
The objective of the present note is to analyze the fluctuations of the magnetization, when either $N^2 p^3$ does not tend to infinity, or $\beta=1$.

\subsection{Main results}
The central results of this note concern the quantity $L_N$ defined in \eqref{eq:def_L_N} and two related quantities, which we will call $L_N^1$ and $L_N^3$ because they appear in the first and the third case of Theorem~\ref{theo:crit_temp_clt}, respectively. They are defined by
\begin{equation}\label{eq:def_TL_N}
L_N^1 := \frac 1{Z_N(\beta)} \sum_{\sigma\in \{-1,+1\}^N} \eee^{-\beta H(\sigma)}  \delta_{\frac 1 {N^{3/4}} \sum_{i=1}^N \sigma_i}
\end{equation}
and
\begin{equation}\label{eq:def_OL_N}
L_N^3 := \frac 1{Z_N(\beta)} \sum_{\sigma\in \{-1,+1\}^N} \eee^{-\beta H(\sigma)}  \delta_{\frac 1 {\sqrt{N^3p^2}} \sum_{i=1}^N \sigma_i}.
\end{equation}

We will analyze the convergence of $L_N$ in the high temperature regime $\beta<1$ and the convergence of $L_N^1$ and $L_N^3$ at the critical temperature. To this end we will furnish the set $\mathcal M(\R)$  of probability measures on $\R$ with the topology of weak convergence. Recall that this weak topology can be metrized by the L\'evy metric
$$
d_{L}(\mu_1,\mu_2) = \inf\{\eps>0\colon \mu_1(-\infty,t-\eps]-\eps \leq \mu_2(-\infty,t] \leq \mu_1(-\infty,t+\eps]+\eps\}.
$$
Endowed with this metric, $\mathcal M(\R)$ becomes a complete, separable metric space.

Our first result is an extension of Theorem 1.1 in \cite{KaLS19} to the regime of smaller values of $p$ such that $pN \to \infty$.

\begin{theorem}\label{theo:high_temp_clt}
Assume that $0<\beta<1$ and let $p=p(N)$ be such that $p N \to\infty$ as $N\to\infty$.
Then, $L_N$, considered as a random element of $\mathcal M(\R)$, converges in probability to $\mathfrak N_{0, 1/(1-\beta)}$.
That is to say, for every $\eps>0$,
$$
\lim_{N\to\infty} \P[d_{L}(L_N, \mathfrak N_{0, 1/(1-\beta)})>\eps] = 0.
$$
\end{theorem}

As can be expected from the results in the Curie-Weiss model and from the fact that at $\beta=1$ the variance of the normal distribution in Theorem \ref{theo:high_temp_clt} explodes, in the critical case $\beta=1$ one has to scale differently. Interestingly, besides this phase transition in $\beta$, we will also find that there is a phase transition in $p$.

\begin{theorem}\label{theo:crit_temp_clt}
\begin{enumerate}
\item
Assume that $\beta=1$ and let $p=p(N)$ be such that $p^4 N^3 \to\infty$ as $N\to\infty$.
Then, $L_N^1$, considered as a random element of $\mathcal M(\R)$, converges in probability to $\mathfrak{M}$.
Here $\mathfrak{M} \in \mathcal M(\R)$ is the probability measure with density
\begin{equation}
\psi(x):=\frac{\eee^{-\frac 1 {12} x^4}}{\int_\R \eee^{-\frac 1 {12} y^4} dy}
\end{equation}
with respect to Lebesgue measure.
The convergence in probability  means that for every $\eps>0$,
$$
\lim_{N\to\infty} \P[d_{L}(L_N^1, \mathfrak M)>\eps] = 0.
$$
\item If $\beta=1$ and $p=p(N)$ is such that $Np \to \infty$, but $p^4 N^3 \to \frac {1}{c^2} \in (0, \infty)$ as $N\to\infty$, then $L_N^1$, considered as a random element of $\mathcal M(\R)$, converges in probability to $\mathfrak{L}_{c}$.
Here $\mathfrak{L}_c \in \mathcal M(\R)$ is the probability measure with  density
\begin{equation}
\psi_c(x):=\frac{\eee^{-c \frac{x^2}{24}-\frac {x^4} {12} }}{\int_\R \eee^{-c \frac{y^2}{24}-\frac {y^4} {12}} dy}
\end{equation}
with respect to Lebesgue measure.
As above, the convergence in probability means that for every $\eps>0$,
$$
\lim_{N\to\infty} \P[d_{L}(L_N^1, \mathfrak{L}_{c})>\eps] = 0.
$$
\item
If again $\beta=1$ and $p=p(N)$ is such that $Np \to \infty$, but $p^4 N^3 \to 0$ as $N\to\infty$, then
$L_N^3$, considered as a random element of $\mathcal M(\R)$, converges in probability to $\mathfrak{N}_{0,12}$.
Again, this is to say
hat for every $\eps>0$,
$$
\lim_{N\to\infty} \P[d_{L}(L_N^3, \mathfrak{N}_{0,12})>\eps] = 0,
$$
and $\mathfrak{N}_{0,12}\in \mathcal M(\R)$ denotes the centered normal distribution with variance~$12$.
\end{enumerate}
\end{theorem}

\section{Technical preparation}
In this section we will prepare for the proof of the main Theorems \ref{theo:high_temp_clt} and \ref{theo:crit_temp_clt}. 
The study of the following function can be motivated when taking the expectation~$\E$ with respect to the randomness
given by $\vep_{i,j}, i,j \in \{1, \ldots, N\}$ of the random variable $Z_N(\beta, g)$ for some function $g \in \mathcal{C}_b(\R)$. Using the independence of the $(\vep_{i,j})_{i,j}$ we will then have to take expectations of the form
\begin{equation}\label{eq:motiv}
\E \eee^{-\beta H(\sigma)}
= \E \left[ \eee^{\gamma \sum_{i,j=1}^N \vep_{i,j} \sigma_i \sigma_j} \right]
= \prod_{i,j=1}^N \E \left[ \eee^{\gamma \vep_{i,j} \sigma_i \sigma_j}\right]
=\prod_{i,j=1}^N\left(1-p + p \eee^{\gamma \sigma_i \sigma_j}\right),
\end{equation}
where $\beta \le 1$ is some fixed value and we set $\gamma:=\frac{\beta}{2Np}$. Noting that $\gamma$ becomes small when $N$ becomes large, we are led to analyze the function $\log (1 - p + p\eee^{z})$ for $z$ small and $0\le p \le 1$. 

More precisely, for an integer $m$ and arbitrary complex variables $z$ and $p$ let
us define the slightly more general functions
\begin{equation}\label{eq:centralfunc}
F_m(x,p, z) := \log (1 - p + p\eee^{xz-m\log \cosh(z)}).
\end{equation}
In this section we do not require that $p$ is a probability and consider $p$ as a complex-valued quantity.
We will next compute the power series expansion of some linear combinations of $F_m(x,p, z)$. Here, the $x$ variable and $m$ will be given and the linear combinations will be expanded in the $p$ and $z$ variables around the origin $(0,0)$.
For any fixed $x$ and $m$, if $|p|< 2$ and $|z|< z_0$ with sufficiently small $z_0>0$, then
$$
|p(\eee^{zx-m\log \cosh(z)}-1)| < 1.
$$
Thus, for fixed $x$ and $m$ the function $F_m(x,p,z)$ is an analytic function of two complex variables $p$ and $z$ on the domain
$$
\mathcal D = \{(p,z)\in\C^2\colon |p| < 2, |z|<z_0\}.
$$
As such, it has a power series expansion which converges uniformly and absolutely on compact subsets of this domain. In particular, by absolute convergence, we can re-arrange and re-group the terms arbitrarily.
For example, the first terms of the power series expansions of $F_1(1,p,z)$, $F_2(2,p,z)$, and $F_2(0,p,z)$  are as follows:
$$
F_1(1,p, z) = p z - \frac{p^2}2 z^2+ \frac{ p(p^2-1)}3 z^3+ \left(\frac {p^2}3-\frac{p^4}{4}\right)z^4+  \mathcal{O}(z^5),
$$
$$
F_2(2,p, z) = 2p z + (p-2p^2) z^2+ \frac{2 p(4p^2-3p-1)}3 z^3- \frac{p(-24p^3-24p^2+5-4)}{6}z^4+  \mathcal{O}(z^5),
$$
and finally
$$
F_2(0,p,z)=-pz^2+\frac 16 p (4-3p)z^4+\mathcal{O}(z^6).
$$
The next lemma collects the expansions which we shall need below.
\begin{lemma}\label{lem:exp_tanh}
We have the following expansions: 
\begin{equation}\label{f1A}
F_1(1,p, z)+F_1(1,p,-z)=  -p^2 \tanh^2 (z) + \mathcal{O}(p^3 z^4) = -p^2 z^2 + \mathcal{O}(p^2 z^4),
\end{equation}
\begin{equation}\label{f6A}
F_1(1,p, z)-F_1(1,p,-z)= 2p\tanh(z)+ \mathcal{O}(p^2 z^3) = 2 p z+ \mathcal{O}(p z^3),
\end{equation}
and
\begin{equation}\label{f7A}
F_2(2,p, z)-F_2(2,p,-z)= 4p\tanh(z)+ \mathcal{O}(p^2 z^3) = 4p z+ \mathcal{O}(p z^3),
\end{equation}
\begin{equation}\label{f8A}
F_2(2,p, z)+F_2(2,p,-z)-2F_2(0,p, z) = 4p(1-p)\tanh^2(z)+ \mathcal{O}(p^3 z^3), 
\end{equation}
\begin{equation}\label{f3A}
F_2(2,p, z)+F_2(2,p,-z)+2F_2(0,p,z)= -4p^2z^2 + \mathcal{O}(p^2z^4).
\end{equation}
\end{lemma}

\begin{proof}
The proofs of all formulae are similar. As an example, we prove~\eqref{f6A}. To this end, we first expand the analytic function $F_1(1,p, z)-F_1(1,p,-z)$ in a Taylor series of the following form
$$
F_1(1,p, z)-F_1(1,p,-z) = \sum_{n=0}^\infty p^n Q_{n}^{(1)} (z).
$$
This expansion converges absolutely and uniformly on compact subsets of $\mathcal D$, and the coefficients $Q_{n}^{(1)} (z)$ are analytic functions on $\{|z|<z_0\}$.
We need to compute the functions $Q_{n}^{(1)} (z)$ for $n=0,1$. This is done by the formula
$$
Q_{n}^{(1)} (z) = \frac{1}{n!} \frac{\dd^n}{\dd p^n} (F_1(1,p, z)-F_1(1,p,-z)) \Big|_{p=0}.
$$
Taking $n=0$ and $n=1$, we obtain that $Q_0^{(1)}(z) = 0$ and
\begin{align*}
Q_{1}^{(1)} (z)
&=\left(\frac{-1+e^{z-\log \cosh z}}{1-p+pe^{z-\log \cosh z}} - \frac{-1+e^{-z-\log \cosh z}}{1-p+pe^{-z-\log \cosh z}}\right) \Big |_{p=0}
=2\tanh z.
\end{align*}
That is, the expansion takes the form
$$
F_1(1,p, z)-F_1(1,p,-z) = 2p \tanh z + \sum_{n=2}^\infty p^n Q_{n}^{(1)} (z).
$$
To see that each $Q_{n}^{(1)} (z)$, $n\geq 2$, does not contain terms of the form $c_0,c_1z,c_2z^2$, we consider the function
$$
G(p,z):= F_1(1,p, z)-F_1(1,p,-z) - 2p \tanh z.
$$
First observe that
$
G(p,0)= 0.
$
Moreover, differentiating with respect to $z$ yields:
{\footnotesize
\begin{align*}
&\frac{\dd}{\dd z} G(p,z)= \frac{8 p^3 \eee^{2 z} \left(\eee^{2 z}-1\right)^2}{\left(\eee^{2 z}+1\right)^2 \left((p-1) \eee^{2 z}-p-1\right) \left((p+1) \eee^{2 z}-p+1\right)},\\
&\frac{\dd^2}{\dd z^2} G(p,z) =
\frac{16 p^3 \eee^{2 z} \left(\eee^{2 z}-1\right) \left(-4 \left(p^2-1\right) \eee^{2 z}-4 \left(p^2-1\right) \eee^{6 z}+\left(p^2-1\right) \eee^{8 z}+2 \left(3 p^2+5\right) \eee^{4 z}+p^2-1\right)}{\left(\eee^{2 z}+1\right)^3 \left((1-p)\eee^{2 z}+p+1\right)^2 \left((p+1) \eee^{2 z}-p+1\right)^2}.
\end{align*}
}
Both derivatives vanish at $z=0$, thus proving that we have an expansion of the form
$$
F_1(1,p, z)-F_1(1,p,-z) = 2p \tanh z +  \sum_{n=2}^\infty p^n z^3 \tilde Q_{n}^{(1)} (z),
$$
where the coefficients $\tilde Q_{n}^{(1)}$ are analytic functions on $\{|z|<z_0\}$.
This proves~\eqref{f6A}.
\end{proof}

\section{Proof of Theorem \ref{theo:high_temp_clt}}\label{sec:3}
In this section we will prove Theorem \ref{theo:high_temp_clt}. The main difference of our following considerations to the proof of the CLT for $\sqrt N m_N$ given in \cite{KaLS19} is that here we will replace the quantity $Z_N(\beta,g)$ as defined in \eqref{ZN(g)} by
a term, which allows an asymptotic expansion also for smaller values of $p$.
To this end, for $\sigma \in \{\pm 1\}^N$ we introduce
\begin{equation}\label{eq:T}
T(\sigma)\colon = \exp\left(\gamma \sum_{i,j=1}^N \vep_{i,j} \sigma_i \sigma_j - \log \cosh(\gamma)\sum_{i,j=1}^N \vep_{i,j}\right)
\end{equation}
and recall that for fixed $\beta <1$ we defined $\gamma:= \frac{\beta}{2 Np}.$
Moreover, let
\begin{equation}\label{tildeZN(g)}
\tilde Z_N(\beta, g):= \sum_{\sigma \in \{-1, +1\}^N} g\left(\frac{|\sigma|}{\sqrt N}\right)T(\sigma)
\end{equation}
for $g \in \mathcal{C}^b(\R), g \ge 0, g \not \equiv 0$.

We will study the behaviour of $ \tilde Z_N(\beta,g)$. To this end, we will introduce the following set of pairs of spins
\begin{eqnarray}\label{eq:typical}
S_N&:=&\left\{(\sigma, \tau) \in \{\pm 1\}^N \times \{\pm 1\}^N: |\sigma|^2 \le N (Np)^{1/5},|\tau|^2\le N (Np)^{1/5},
\right. \nonumber \\
&&\quad \left. |\sigma \tau|^2 \le N (Np)^{1/5}\right\}.
\end{eqnarray}

Here we set $|\sigma \tau|:= \sum_{i=1}^N \sigma_i \tau_i$. The spins in $S_N$ will be called the {\it typical} spins in the rest of this section. The spins in the complement $S_N^c:= (\{\pm 1\}^N \times \{\pm 1\}^N) \setminus S_N$ are called {\it atypical}.
In a slight abuse of notation, we say that a configuration $\sigma \in \{\pm 1\}^N$ is typical and write
$\sigma \in S_N$. By this we will mean that $|\sigma|^2 \le N (Np)^{1/5}$.

We start with the following computation
\begin{lemma}\label{ET}
For all $p=p(N)$ such that $pN \to \infty$ and all $\sigma\in S_N$ we have
\begin{equation}\label{eq:ET}
\E T(\sigma)=
\exp\left( -\frac{\beta^2}{8}+\frac{\beta}{2N}|\sigma|^2+o(1)\right)
\end{equation}
with $o(1)$-term that is uniform over $\sigma \in S_N$.
\end{lemma}

\begin{proof}
Similarly to the computation of \eqref{eq:motiv} we compute
$$
\E T(\sigma)
= \prod_{i,j=1}^N \E \left[ \eee^{\gamma \vep_{i,j} \sigma_i \sigma_j-\vep_{i,j}\log \cosh(\gamma)}\right]
=\prod_{i,j=1}^N\left(1-p + p \eee^{\gamma \sigma_i \sigma_j-\log \cosh(\gamma)}\right).
$$
Defining $$f(x) = f(x; p,\gamma) = \log (1-p + p\eee^{\gamma x-\log \cosh(\gamma)})=F_1(x,p,\gamma),$$
with $F_1$ defined in \eqref{eq:centralfunc} we can write
$$
\E T(\sigma)
=
\exp\left(\sum_{i,j=1}^N \log (1-p + p\eee^{\gamma \sigma_i \sigma_j-\log \cosh(\gamma)})\right)
=\exp\left(\sum_{i,j=1}^N f(\sigma_i \sigma_j)\right).
$$
Observe that since $\sigma_i\in \{\pm 1\}$ for all $i$, we are interested in the behaviour of $f$ only in the two values $+1$ and $-1$.
For these two values we can rewrite $f$ in a linear form. This means that  we write
$$
f(x) = a_0 + a_1 x, \qquad x \in \{-1,+1\}.
$$
Here $a_0$ and $a_1$ depend on $p$ and $\gamma$, of course, and are given by
\begin{align*}
a_0
&= \frac  {f(1) + f(-1)}{2} =\frac{F_1(1,p,\gamma)+F_1(1,p,-\gamma)}{2}
,\\
a_1
&=
\frac {f(1)-f(-1)}2=\frac{F_1(1,p,\gamma)-F_1(1,p,-\gamma)}{2}
\end{align*}
because $F_1(-1,p,\gamma)=F_1(1,p,-\gamma)$.
Thus
$$
\E  T(\sigma)  =
\exp\left(\sum_{i,j=1}^N f(\sigma_i \sigma_j)\right)=
\exp\left(\sum_{i,j=1}^N (a_0+a_1\sigma_i \sigma_j)\right)
=
\exp\left(N^2 a_0+a_1|\sigma|^2\right).
$$
Recalling~\eqref{f1A} we obtain
$$
N^2a_0= -N^2\frac 12 p^2 \gamma^2+\mathcal{O}(N^2 p^2 \gamma^4)=-\frac{\beta^2}{8}+o(1).
$$
On the other hand,
by \eqref{f6A}
$$
a_1 |\sigma|^2 = (p \gamma+\mathcal{O}(p \gamma^3))|\sigma|^2= \frac{\beta}{2N}|\sigma|^2+o(1)
$$
with an $o(1)$-term that is uniform over all $\sigma \in S_N$. Indeed, for a typical configuration $\sigma$ we have
$$
p \gamma^3|\sigma|^2= \frac{p\beta^3}{8 N^3 p^3}|\sigma|^2\le
\frac{\beta^3}{8 N^3 p^3} (Np)^{\frac 65}
= \frac{\beta^3}{8}\frac {1}{(Np)^{\frac{9}{5}}} \to 0.
$$
This proves the claim.
\end{proof}

Since eventually we want to compare $\V \tilde Z_N(\beta,g)$ to $[\E \tilde Z_N(\beta,g)]^2$, in a next step, we will compute $\E (T(\sigma) T(\tau))$ for typical $\sigma$ and $\tau$.

\begin{lemma}\label{ETsTt}
For $p=p(N)$ such that $pN \to \infty$ and any $(\sigma, \tau) \in S_N$ we have
\begin{equation}
\E\left[ T(\sigma)T(\tau)\right]
=
\exp\left(-\frac{\beta^2}{4}+\frac{\beta}{2N}(|\sigma|^2+|\tau|^2)+o(1)
\right)
\end{equation}
uniformly over all $(\sigma, \tau) \in S_N$.
\end{lemma}

\begin{proof}
The idea of the proof is similar to the proof of Lemma \ref{ET}.
Taking the expectations as in the proof of Lemma \ref{ET} we get
\begin{eqnarray*}
\E\left[ T(\sigma) T(\tau)\right]&=&
\prod_{i,j=1}^N \E \left[ \eee^{\gamma (\sigma_i \sigma_j+\tau_i \tau_j)\vep_{i,j} -2\vep_{i,j} \log \cosh(\gamma)}\right]
= \E\exp\left(\sum_{i,j=1}^N g(\sigma_i \sigma_j+ \tau_i \tau_j)\right),
\end{eqnarray*}
where we have set
$$
g(x):= \log(1-p+p\eee^{\gamma x-2 \log \cosh (\gamma)})= F_2(x,p,\gamma)
$$
with $F_2$ defined as in \eqref{eq:centralfunc}.

Note that we are just interested in $g(0)$ and  $g(\pm 2)$ since both, $\sigma$ and $\tau$, are in $\{\pm 1\}^N$ and hence we have that $\sigma_i \sigma_j+\tau_i \tau_j \in \{-2,0,2\}$.
For these values, the function $g$ can be represented in the form
$$
g(x_1+x_2)= b_0 + b_1 x_1 + b_2 x_2 + b_{12} x_1 x_2,
\quad x_1,x_2\in \{-1,+1\},
$$
for coefficients $b_0,b_1,b_2,b_{12}$ that again depend on $p$ and $\gamma$.
They can be computed by solving a system of
four linear equations in four variables:
\begin{align*}
b_0 + b_1+b_2 +b_{12} =& g(2),\\
b_0 + b_1-b_2 -b_{12} =& g(0),\\
b_0 - b_1+b_2 -b_{12} =& g(0),\\
b_0 - b_1-b_2 +b_{12} =& g(-2).
\end{align*}
This system of equations has a unique solution
\begin{align*}
b_0 &= \frac {g(2)+g(-2)+2g(0)}4=\frac {F_2(2,p,\gamma) + F_2(2,p,-\gamma)+ 2F_2(0,p,\gamma)}4\\
b_{12} &= \frac {g(2)+g(-2)-2g(0)}4 =\frac {F_2(2,p,\gamma) + F_2(2,p,-\gamma)- 2F_2(0,p,\gamma)}4 , \\
b_1
&= b_{2} = \frac {g(2)-g(-2)}4 =\frac {F_2(2,p,\gamma) - F_2(2,p,-\gamma)}4.
\end{align*}
Thus
$$
\sum_{i,j=1}^N g(\sigma_i \sigma_j+ \tau_i \tau_j)= N^2 b_0 +b_1(|\sigma|^2+|\tau|^2)+b_{12} |\sigma \tau|^2.
$$
Using \eqref{f7A},\eqref{f8A},\eqref{f3A}, we see that
$$
b_0= -p^2\gamma^2+\mathcal{O}(p^2\gamma^4) \quad\mbox{and hence } N^2 b_0= -N^2p^2\gamma^2+\mathcal{O}(N^2p^2\gamma^4)= -\frac{\beta^2}4+o(1),
$$
$$
b_{12}=\mathcal{O}(p\gamma^2) \quad\mbox{and hence } b_{12}|\sigma \tau|^2=\mathcal{O}(p\gamma^2 N (Np)^{1/5})=\mathcal{O}((Np)^{-4/5})=o(1)
$$
for all typical $(\sigma, \tau) \in S_N$, where the $o(1)$-terms are uniform.
And finally,
$$
b_1=p\gamma +\mathcal{O}(p\gamma^3) \quad\mbox{and hence } b_{1}|\sigma|^2=\frac{\beta}{2N}|\sigma|^2+\mathcal{O}(|\sigma|^2 p\gamma^3)
=\frac{\beta}{2N}|\sigma|^2+o(1)
$$
as well as $b_{1}|\tau|^2=\frac{\beta}{2N}|\tau|^2+o(1)$ for all typical $\sigma$ and $\tau$.
This proves the claim.
\end{proof}

Combining Lemma \ref{ET} and \ref{ETsTt} yields
\begin{corollary}\label{cor:cov}
Uniformly over all $(\sigma, \tau)\in S_N$ we have
\begin{equation}
\mathrm{Cov}(T(\sigma)T(\tau))=o(\E(T(\sigma)\E(T(\tau)).
\end{equation}
\end{corollary}
\begin{proof}
This follows directly from the definition of the covariance and Lemmas \ref{ET} and \ref{ETsTt}.
\end{proof}
Note that the claim would not be true if would drop the $\log \cosh \gamma$ term in the definition of $T(\sigma)$. The role of this term is to make $T(\sigma)$'s ``asymptotically independent''.

Corollary \ref{cor:cov} already suggests that we should have $\V \tilde Z_N(\beta, g) = o(\E (\tilde Z_N(\beta,g))^2)$ for all $g \in \mathcal{C}^b(\R), g \ge 0, g \not\equiv 0$. However, to prove this we still need to control the contribution of the atypical configurations. This is done in the following propositions.

The first of them may be interesting in its own rights, moreover, its proof is also a nice warm-up for the proof of Proposition
\ref{prop:var}  thereafter, which is similar, but technically more demanding.

\begin{proposition}\label{prop:expect}
For all $g \in \mathcal{C}^b(\R), g \ge 0, g \not\equiv 0$ we have
\begin{equation}
\lim_{N \to \infty} \frac{\E \tilde Z_N(\beta, g)}{2^N \eee^{-\frac{\beta^2}{8}} \E_{\xi} [g(\xi)\eee^{\frac {\beta}{2}\xi^2}]}=1.
\end{equation}
Here $\xi$ denotes a standard normally distributed random variable.
\end{proposition}

\begin{proof}
By an obvious decomposition of $\{\pm 1\}^N$ into typical and atypical $\sigma$'s, we have
\begin{equation}\label{eq:expect_typical_atypical}
\E \tilde Z_N(\beta, g)
=
\sum_{\sigma \in S_N}g\left(\frac{|\sigma|}{\sqrt N}\right)\E T(\sigma)
+
\sum_{\sigma \in S_N^c} g\left(\frac{|\sigma|}{\sqrt N}\right)\E T(\sigma).
\end{equation}

For $\sigma\in S_N$ we can use the result of Lemma~\ref{ET} to obtain:
\begin{eqnarray*}
\sum_{\sigma \in S_N}g\left(\frac{|\sigma|}{\sqrt N}\right)T(\sigma)=
\sum_{\sigma \in S_N}g\left(\frac{|\sigma|}{\sqrt N}\right) \exp\left( -\frac{\beta^2}{8}+\frac{\beta}{2N}|\sigma|^2+o(1)\right)
\end{eqnarray*}
with $o(1)$-term that is uniform over $S_N$.

For atypical $\sigma$'s we claim that
\begin{equation}\label{eq:claim_untypical}
\sum_{\sigma \in S_N^c} g\left( \frac{|\sigma|}{\sqrt N} \right) \E T(\sigma)
=  o\left(2^N\right).
\end{equation}
Letting $\|g\|_\infty:= \sup_{t\in\R} |g(t)| < \infty$ and using the same arguments as in the proof of Lemma~\ref{ET}, we can write
\begin{align*}
\left|\sum_{\sigma \in S_N^c} g\left( \frac{|\sigma|}{\sqrt N} \right)\E T(\sigma)\right|
&\leq
\|g\|_\infty \sum_{\sigma \in S_N^c}\E T(\sigma)\\
&\leq
\|g\|_\infty \sum_{\sigma \in S_N^c} \eee^{-\frac{\beta^2}8+ o(1)+ \frac{\beta}{2N}|\sigma|^2+ \frac{C_{N,1}}{N^2p^2}\frac{|\sigma|^2}N}
\end{align*}
for some sequence of constants $C_{N,1}$ that does not depend on $\sigma$ and stays bounded. Hence
\begin{align*}
\left|\sum_{\sigma \in S_N^c} g\left( \frac{|\sigma|}{\sqrt N} \right)\E T(\sigma)\right|&\le
\|g\|_\infty \eee^{-\frac{\beta^2}8+ o(1)} \sum_{k\in\Z : k^2>N (Np)^{\frac 15}}
\eee^{\frac{\beta}{2N} k^2 + \frac{C_{N,1}}{N^2p^2}\frac{k^2}N} \nu_N(k),
\end{align*}
where $\nu_N(k)$ is the number of $\sigma\in \{-1,+1\}^N$ such that $|\sigma| = k$. By the Local Limit Theorem, we have
$$
2^{-N}\nu_N(k) \leq \frac{C}{\sqrt N} \eee^{-\frac{k^2}{2N}}, \quad k\in\Z,
$$
for some absolute constant $C$, which here and in the sequel may change from line to line. By these considerations we arrive at
\begin{align*}
\left|\sum_{\sigma \in S_N^c} g\left( \frac{|\sigma|}{\sqrt N} \right) \E T(\sigma)\right|
&\leq C \|g\|_\infty
 \eee^{-\frac{\beta^2}8+ o(1)} \frac {2^N} {\sqrt N} \sum_{k\in\Z : k^2>N (Np)^{\frac 15}}
\eee^{\frac{k^2}{2N} \left(\beta-1 + \frac{2C_{N,1}}{N^2p^2}\right)}.
\end{align*}
Since $\beta<1$, and $C_{N,1}$ stays  bounded, there is $\delta>0$ such that
$
\beta-1 + \frac{2C_{N,1}}{N^2p^2} < -\delta
$
for all sufficiently large $N$. Therefore,
\begin{align*}
\left|\sum_{\sigma \in S_N^c} g\left( \frac{|\sigma|}{\sqrt N} \right) \E T(\sigma)\right|
&\leq  C \|g\|_\infty
 \eee^{-\frac{\beta^2}8+ o(1)} \frac {2^N} {\sqrt N} \sum_{k\in\Z : k^2>N (Np)^{\frac 15}}
\eee^{-\delta \frac{k^2}{2N}}
\\
&\leq C \|g\|_\infty
 \eee^{-\frac{\beta^2}8+ o(1)} 2^N \int_{(Np)^{1/11}}^{\infty} \eee^{-\delta t^2/2} d  t,
\end{align*}
by approximating the Riemann sum by an integral over a larger domain.
This proves the claim \eqref{eq:claim_untypical} since $Np \to +\infty$ as $N\to\infty$.

Finally, we prove that
\begin{equation}\label{eq:sum_g_exp_de_Moivre}
\lim_{N\to\infty} \frac 1 {2^N} \sum_{\sigma \in \{-1,+1\}^N} g\left( \frac{|\sigma|}{\sqrt N} \right)  \eee^{\frac{\beta}{2N}|\sigma|^2}
= \E_{\xi}[g(\xi)\eee^{\frac {\beta}{2}\xi^2}],
\end{equation}
where $\xi$ is a standard normal random variable. As in \cite{KaLS19}, proof of Theorem 1.4., this follows from
the Central Limit Theorem of de Moivre--Laplace applied to
$\frac{|\sigma|}{\sqrt N}$ together with the uniform integrability of the sequence
$g\left( \frac{|\sigma|}{\sqrt N} \right)  \eee^{\frac{\beta}{2N}|\sigma|^2}$ for $\beta<1$ (cf. ~\cite{Ellis-EntropyLargeDeviationsAndStatisticalMechanics}, Proof of Theorem V.9.4),
\end{proof}

We are now ready to prove that our guess about the size of the variance of $\tilde Z_N(\beta,g)$ was correct.
\begin{proposition}\label{prop:var}
For all $g \in \mathcal{C}^b(\R), g \ge 0, g \not\equiv 0$ we have
$$
\V \tilde Z_N(\beta,g) = o((\E \tilde Z_N(\beta,g))^2).
$$
\end{proposition}

\begin{proof}
The idea of the proof is similar to the previous one.
Again we decompose the term of interest:
\begin{eqnarray}\label{eq:var_typical_atypical}
\V \tilde Z_N(\beta, g)
&=&
\sum_{(\sigma,\tau) \in S_N}g\left(\frac{|\sigma|}{\sqrt N}\right)g\left(\frac{|\tau|}{\sqrt N}\right)\mathrm{Cov}(T(\sigma),T(\tau))
\nonumber \\
&& \qquad+
\sum_{(\sigma,\tau) \in S_N^c} g\left(\frac{|\sigma|}{\sqrt N}\right)g\left(\frac{|\tau|}{\sqrt N}\right)\mathrm{Cov}(T(\sigma),T(\tau)).
\end{eqnarray}
We already saw in Corollary \ref{cor:cov} that for $(\sigma,\tau) \in S_N$ we have that
$\mathrm{Cov}(T(\sigma)T(\tau))=o(\E(T(\sigma)\E(T(\tau))$.  This, together with the fact that $g$ is bounded implies that
\begin{multline}\label{eq:var_main_term}
\sum_{(\sigma,\tau) \in S_N}g\left(\frac{|\sigma|}{\sqrt N}\right)g\left(\frac{|\tau|}{\sqrt N}\right)\mathrm{Cov}(T(\sigma),T(\tau))
\\
=o\left(\sum_{(\sigma,\tau) \in S_N}g\left(\frac{|\sigma|}{\sqrt N}\right)g\left(\frac{|\tau|}{\sqrt N}\right)\E(T(\sigma))\E(T(\tau))\right).
\end{multline}
On the other hand, since $g \ge 0$,
\begin{multline*}
\sum_{(\sigma,\tau) \in S_N}g\left(\frac{|\sigma|}{\sqrt N}\right)g\left(\frac{|\tau|}{\sqrt N}\right)\E (T(\sigma))\E (T(\tau))
\\
\leq \sum_{\sigma,\tau \in \{\pm1\}^N}g\left(\frac{|\sigma|}{\sqrt N}\right)g\left(\frac{|\tau|}{\sqrt N}\right)\E(T(\sigma))\E(T(\tau))
\end{multline*}
and the right hand side is $\E \tilde Z_N(\beta,g))^2$.

From the proof of Proposition \ref{prop:expect} we see that for any $\sigma, \tau$
\begin{equation}
\E [T(\sigma)] \E [T(\tau)]
=
\eee^{-\frac{\beta^2}4 + o(1)+\frac{\beta}{2} \frac {|\sigma|^2+|\tau|^2} N + \frac{C_{N,1}} {N^2p^2} \frac {|\sigma|^2 + |\tau|^2}N}, \label{eq:cov_computation1}
\end{equation}
where we recall that the sequence of constants $C_{N,1}$ does not depend on $\sigma$ and $\tau$ and is bounded.

Similarly, along the lines of the proof of Lemma \ref{ETsTt} we see that
\begin{align}
\E [T(\sigma)T(\tau)]
&=
\eee^{-
\frac{\beta^2}4+o(1) +\frac{|\sigma\tau|^2}{N}\frac{C_{N,2}}{Np}
+
\frac{|\sigma|^2+|\tau|^2}N \left(\frac \beta 2 + \frac{C_{N,3}}{N^2p^2} \right)
},\label{eq:cov_computation2}
\end{align}
where also the sequences of constants $C_{N,2}$ and $C_{N,3}$ do not depend on $\sigma$ and $\tau$ and stay bounded.

To treat the second sum on the right hand side of \eqref{eq:var_typical_atypical}
denote by $V_N(k,l,m)$ the set of pairs
$$(\sigma,\tau)\in \{-1,+1\}^N\times\{-1,+1\}^N \quad \mbox{for which }|\sigma|=k,|\tau|= l, \mbox{and }|\sigma \tau| = m.$$
Moreover, by $\nu_N(k,l,m) = \# V_N(k,l,m)$ we denote the number of such pairs.
Again, we want to apply a Local Limit Theorem. To this end, if $\sigma$ and $\tau$ are taken independently and uniformly at random from $\{-1,+1\}^N$, the vectors $(\sigma_i,\tau_i,\sigma_i\tau_i)$, $1\leq i\leq N$, are i.i.d.
Their mean is $0$ and their covariance matrix is the $3\times 3$ identity matrix because
$$
\sigma_i (\sigma_i\tau_i) = \tau_i, \quad \tau_i (\sigma_i\tau_i) = \sigma_i, \quad \sigma_i^2 = \tau_i^2 = (\sigma_i\tau_i)^2 = 1.
$$
The three-dimensional Local Central Limit Theorem~\cite{Davis1995} tells us that there is a universal constant $C$ such that
$$
\nu_N(k,l,m) \leq C 2^{2N} N^{-3/2} \eee^{-\frac{k^2}{2N} - \frac{l^2}{2N} - \frac{m^2}{2N}}, \quad (k,l,m)\in\Z^3.
$$
Combining this with ~\eqref{eq:cov_computation2} we see that for any $(k,l,m)\in\Z^3$,
\begin{multline*}
2^{-2N}\sum_{(\sigma,\tau) \in V_N(k,l,m)}
\E [T(\sigma) T(\tau)]
\leq \\
CN^{-3/2}
\exp\left(-
\frac{\beta^2}4+o(1) +\frac{m^2}{2N}\left(\frac{2C_{N,2}}{Np}-1\right)
+
\frac{k^2+l^2}N \left(\frac {\beta-1} 2 + \frac{C_{N,3}}{N^2p^2} \right)
\right).
\end{multline*}

Since $\beta<1$, we obtain for some $\delta >0$ and $N$ large enough
\begin{equation*}
2^{-2N}\sum_{(\sigma,\tau) \in V_N(k,l,m)}
\E [T(\sigma) T(\tau)]
\leq
CN^{-3/2}
\eee^{-\frac{\beta^2}4+o(1)}
\eee^{-\delta \frac{k^2+l^2 + m^2} {2N}}.
\end{equation*}
In a very similar way we can bound the sum of the $\E [T(\sigma)]\E[T(\tau)]$-terms over $(\sigma,\tau) \in V_N(k,l,m)$.
Employing the same notation as in the previous step we obtain
\begin{align*}
&2^{-2N}\sum_{(\sigma,\tau) \in V_N(k,l,m)}
\E [T(\sigma)] \E[T(\tau)]\leq C   N^{-3/2} \eee^{-\frac{k^2}{2N} - \frac{l^2}{2N} - \frac{m^2}{2N}-\frac{\beta^2}{4}+o(1) + \frac{\beta}{2} \frac {k^2+l^2} N + \frac {C_{N,1}}{N^2p^2}  \frac {k^2 + l^2}N}
\\
&\leq CN^{-3/2}
\eee^{-\frac{\beta^2}{4}+o(1)}
\eee^{-\delta \frac{k^2+l^2 + m^2} {2N}}
\end{align*}
for some $\delta >0$ and $N$ sufficiently large.

We can therefore conclude,
$$
2^{-2N}\sum_{(\sigma,\tau) \in V_N(k,l,m)}
\left| \Cov (T(\sigma), T(\tau))\right|
\leq
CN^{-3/2}
\eee^{-\frac{\beta^2}{4}+o(1)}
\eee^{-\delta \frac{k^2+l^2 + m^2} {2N}}
$$
for some $\delta >0$ and $N$ sufficiently large.
Recalling that the function $g$ is bounded, we arrive at
\begin{multline*}
\left|\sum_{(\sigma,\tau) \in S_N^c}
g\left(\frac{|\sigma|}{\sqrt N}\right) g\left(\frac{|\tau|}{\sqrt N}\right)
 \Cov (T(\sigma), T(\tau)) \right|
\\
\leq
C \|g\|_\infty^2 2^{2N} \eee^{-\frac{\beta^2}{4}+o(1)}  N^{-3/2}
\sum_{\substack{(k,l,m)\in\Z^3\\N^{-1/2}(k,l,m)\in D_N}} \eee^{-\delta \frac{k^2+l^2 + m^2} {2N}},
\end{multline*}
where $D_N:=\{(x,y,z)\in\R^3\colon |x|>(Np)^{\frac 1{10}} \text{ or } |y|>(Np)^{\frac 1{10}} \text{ or } |z|>(Np)^{\frac 1{10}} \}$.
The sum on the right hand side can again be considered as a Riemann sum, which can be bounded from above by the corresponding
integral over a larger domain. Including the pre-factor $N^{-3/2}$ this yields
$$
 N^{-3/2}
\sum_{\substack{(k,l,m)\in\Z^3\\N^{-1/2}(k,l,m)\in D_N}} \eee^{-\delta \frac{k^2+l^2 + m^2} {2N}}
\leq
\int_{\frac 13 D_N} \eee^{-\delta \frac{x^2+y^2 + z^2} {2}} d x\, d y \,  d z  = o(1),
$$
because $Np \to \infty$ as $N\to\infty$. Combining these considerations with Proposition \ref{prop:expect} we obtain
$$
\sum_{(\sigma, \tau) \in S_N^c} g\left(\frac{|\sigma|}{\sqrt N}\right) g\left(\frac{|\tau|}{\sqrt N}\right) \Cov (T(\sigma), T(\tau))
=
o(1)\cdot (\E \tilde Z_N(\beta, g))^2.
$$
Finally, the assertion follows from a combination of the estimates for the typical configurations and the atypical configurations.
\end{proof}

We are now ready to prove Theorem \ref{theo:high_temp_clt}.

\begin{proof}[Proof of Theorem \ref{theo:high_temp_clt}]
Proposition \ref{prop:var} shows that, as long as $Np\to \infty$, we have that $\V(\tilde Z_N(\beta,g))= o((\E \tilde Z_N(\beta, g))^2)$ for all non-negative $g \in \mathcal{C}_b(\R)$, $g\not  \equiv 0$.
This immediately implies
\begin{equation}\label{eq:in_probab}
\frac{ \tilde Z_N(\beta, g)}{\E \tilde Z_N(\beta, g)} \to 1
\end{equation}
in $L^2$, for all non-negative $g \in \mathcal{C}_b(\R)$, $g\not  \equiv 0$.

By Chebyshev's inequality, this convergence holds in probability, as well.

Moreover, consider the quantity $Z_N(\beta,g)$ defined in \eqref{ZN(g)}. Note that
\begin{eqnarray}\label{eq:ZN_tilde_ZN}
Z_N(\beta,g)&=& \sum_{\sigma \in \{-1, +1\}^N} \eee^{-\beta H(\sigma)} g\left( \frac{\sum_{i=1}^N \sigma_i}{\sqrt N}\right)\nonumber \\
&=& \sum_{\sigma \in \{-1, +1\}^N}g\left( \frac{\sum_{i=1}^N \sigma_i}{\sqrt N}\right)T(\sigma) \eee^{\sum_{i,j} \vep_{i,j} \log \cosh(\gamma)}\nonumber \\
&=& \tilde Z_N(\beta,g)\eee^{\sum_{i,j} \vep_{i,j} \log \cosh(\gamma)}.
\end{eqnarray}

Recall that for the convergence of the random probability measure $L_N$ defined in~\eqref{eq:def_L_N}
we need to consider its integral against all non-negative $g \in \mathcal{C}_b(\R)$.
However, we have that
$$
\int_{-\infty}^{+\infty} g(x) L_N(d x)
=
\E_{\mu_\beta }\left[ g\left( \frac{\sum_{i=1}^N \sigma_i}{\sqrt N} \right)\right]
=
\frac{Z_N(\beta, g)}{Z_N(\beta)}.
$$
But from \eqref{eq:ZN_tilde_ZN} we obtain
$$
\frac{Z_N(\beta, g)}{Z_N(\beta)}= \frac{\tilde Z_N(\beta, g)}{\tilde Z_N(\beta,1)}
$$

It follows from~\eqref{eq:in_probab} and Proposition \ref{prop:expect} that
$$
\lim_{N\to\infty}\int_{-\infty}^{+\infty} g(x) L_N(d x)
=
\lim_{N\to\infty} \frac{\E Z_N(\beta, g)}{\E Z_N(\beta)}
=
\sqrt{1-\beta}\, \E_\xi [g(\xi) \eee^{\frac \beta 2 \xi^2}]
$$
in probability, where $\xi$ is a standard normally distributed random variable.

However, the right-hand side of this equation is nothing but $\int_{-\infty}^{+\infty} g(x) \phi_{0, \frac 1 {1-\beta}}(x) dx$, where
$\phi_{0, \frac 1 {1-\beta}}(x)$ is the density of a normal distribution with mean $0$ and variance $1/(1-\beta)$.
Hence we have shown that $L_N$, considered as a random element of the space of probability measures $\mathcal M(\R)$, converges in probability to a normal distribution with mean $0$ and variance $1/(1-\beta)$, considered as a deterministic point in $\mathcal M(\R)$. This was the assertion of Theorem \ref{theo:high_temp_clt}.
\end{proof}

\section{Proof of Theorem \ref{theo:crit_temp_clt}}
The core idea in the proof of Theorem \ref{theo:crit_temp_clt} is similar to the proof in the previous section. Therefore, we will use most of the definitions introduced there.
However, as can be seen from the result, the relevant spin configurations are no longer those, where $|\sigma|^2$ is of order $N$ (and the same for $|\tau|^2$). We therefore introduce two new sets of configurations:
\begin{eqnarray}\label{eq:typical_alt1}
R_N^1&:=&\left\{(\sigma, \tau) \in \{\pm 1\}^N \times \{\pm 1\}^N: |\sigma|^2 \le N^{3/2} h(Np),|\tau|^2\le N^{3/2} h(Np),
\right. \nonumber \\
&&\quad \left. |\sigma \tau|^2 \le N (Np)^{1/5}\right\}.
\end{eqnarray}
Here $h(x)$ is an increasing function going to infinity in such a way that $\frac{h(Np)}{N^{3/2}p^2} \to 0$
and $(Np)^{1/10} / h(Np) \to \infty$ as $N \to \infty$.
The regime $R_N^1$ will be relevant, whenever $\beta=1$ and $p^4 N^3 \not \to 0$.

In the regime when $p^4 N^3  \to 0$ we define
\begin{eqnarray}\label{eq:typical_alt2}
R_N^2&:=&\left\{(\sigma, \tau) \in \{\pm 1\}^N \times \{\pm 1\}^N: |\sigma|^2 \le N (Np)^{2}h'(Np),|\tau|^2\le N (Np)^{2}h'(Np),
\right. \nonumber \\
&&\quad \left. |\sigma \tau|^2 \le N (Np)^{1/5}\right\}.
\end{eqnarray}
Here $h'(x)$ is an increasing function going to infinity in such a way that
$$
N (Np)^{2}h'(Np) = o(N^{3/2}),\;\; h'(Np) = o((Np)^2) \quad \mbox{ and }\eee^{-\frac 1{120} h'(Np)} =o(N^{3/4}p).
$$
We are able to find $h'$ such that the third condition is fulfilled, because we can take $h'(Np)=\frac{N^{3/2}}{N^3 p^2}h''(Np)= \frac{1}{N^{3/2} p^2}h''(Np)$ with an $h''$ that goes to $0$ arbitrarily slowly. On the other hand, indeed $\eee^{-\frac{1}{N^{3/2} p^2}} =o(N^{3/4}p)$.
Finally, we also require that $\frac{(Np)^2}{\sqrt N} h'(Np)=o( (Np)^{\frac 1{10}})$. This is also possible, because $R_N^2$ becomes relevant, when $p^4 N^3  \to 0$, hence, when $p =o(N^{-\frac 34})$, hence $Np=o(N^{-\frac 14})$. So again, setting
$h'(Np)=\frac{N^{3/2}}{N^3 p^2}h''(Np)$ we require that $h''(Np)$ grows slower than $(Np)^{\frac 1{10}}$, which is in agreement with our above choices.

$R_N^1$ will be the set of the typical pairs of spin configurations for the proof of the first and the second part of the theorem, while $R_N^2$ will play the same role in the proof of the third part of the theorem.
Thus, whenever we consider the set $R_N^1$ we will tacitly assume that $p^4 N^3 \not \to 0$, while, when we use $R_N^2$, we will think of $p$ being such that $p^4 N^3 \to 0$.
So these sets will replace the set $S_N$ from the previous section.

For the rest of this section we will keep $\beta=1$ fixed.

Our first step is an analogue  of Lemma \ref{ET}. Note, however, that the error term in the subsequent lemma is uniform over all $\sigma$'s rather than only over the typical ones.

\begin{lemma}\label{ETcrit}
Let $\beta=1$. For all $p=p(N)$ such that $pN \to \infty$ and all $\sigma \in\{-1,+1\}^N$ we have  
\begin{equation}\label{eq:ETcrit}
\E T(\sigma)=\exp\left( -\frac{1}{8}+p \tanh(\gamma)|\sigma|^2 + o(1)\right)
\end{equation}
with $o(1)$-term that is uniform over $\sigma \in \{-1,+1\}^N$. 
\end{lemma}
\begin{proof}[Proof of Lemma \ref{ETcrit}]
The proof is almost verbatim the same as the proof of Theorem~\ref{ET}. We will follow this proof and write
$$
\E T(\sigma)
=\exp\left(\sum_{i,j=1}^N f(\sigma_i \sigma_j)\right)
$$
with
$$
f(x)= \log (1-p + p\eee^{\gamma x-\log \cosh(\gamma)})=F_1(x,p,\gamma).
$$
Since $\sigma_i\in \{\pm 1\}$ for all $i$, we can rewrite $f$ as
$$
f(x) = a_0 + a_1 x, \qquad x \in \{-1,+1\}.
$$

The only difference is that we expand the expression
$$
a_1
=\frac{F_1(1,p,\gamma)-F_1(1,p,-\gamma)}{2}.
$$
using Lemma \ref{lem:exp_tanh}, in particular, \eqref{f6A}. This gives
$$
a_1 |\sigma|^2 = p \tanh(\gamma)|\sigma|^2  + \mathcal{O}(p^2 \gamma^3 |\sigma|^2 )
$$
Let us have a closer look at the $\mathcal{O}(p^2 \gamma^3 |\sigma|^2 )$-term of this expansion. Since $|\sigma|^2$ is at most $N^2$ and $\gamma=\frac{1}{2Np}$ for $\beta=1$ we see that
$p^2 \gamma^3 |\sigma|^2 \le \frac 1 {8 Np}$ which goes to $0$ uniformly for all $\sigma \in \{-1,+1\}^N$.
\end{proof}
The following is the analogue of Lemma \ref{ETsTt}. Again, the error term is uniform over all pairs $(\sigma,\tau)\in \{\pm 1\}^N \times \{\pm 1\}^N$.

\begin{lemma}\label{ETsTtcrit}
Let $\beta=1$. For $p=p(N)$ such that $pN \to \infty$ and any $\sigma, \tau \in \{-1,+1\}^N$ 
\begin{equation*}
\E\left[ T(\sigma)T(\tau)\right]
=
\exp\left(-\frac{1}{4}+ p \tanh(\gamma)(|\sigma|^2+|\tau|^2)+p(1-p)\tanh^2(\gamma)|\sigma\tau|^2+o(1)
\right)
\end{equation*}
uniformly over all $(\sigma, \tau) \in \{-1,+1\}^N$. 
\end{lemma}

\begin{proof}[Proof of Lemma \ref{ETsTtcrit}]
The proof follows the idea of the proof of Lemma \ref{ETsTt}. We will apply similar changes as those that were made when going from the proof of Lemma \ref{ET} to the proof of Lemma \ref{ETcrit}.
Following the proof of Lemma \ref{ETsTt} we again write
\begin{eqnarray*}
\E\left[ T(\sigma) T(\tau)\right]&=&
\prod_{i,j=1}^N \E \left[ \eee^{\gamma (\sigma_i \sigma_j+\tau_i \tau_j)\vep_{i,j} -2\vep_{i,j} \log \cosh(\gamma)}\right]
= \E\exp\left(\sum_{i,j=1}^N g(\sigma_i \sigma_j+ \tau_i \tau_j)\right),
\end{eqnarray*}
with
$$
g(x):= \log(1-p+p\eee^{\gamma x-2 \log \cosh (\gamma)})= F_2(x,p,\gamma).
$$
We again represent $g$ in the form
$$
g(x_1+x_2)= b_0 + b_1 x_1 + b_2 x_2 + b_{12} x_1 x_2,
\quad x_1,x_2\in \{-1,+1\}.
$$

We now expand the $b_1$ and $b_2$ using \eqref{f7A} and \eqref{f8A} to obtain
$$
b_1=b_2=p\tanh(\gamma) +\mathcal{O}(p^2\gamma^3)
$$
and
$$
b_{12}= p(1-p)\tanh^2(\gamma)+ \mathcal{O}(p^2\gamma^3)
$$
Again, $p^2 \gamma^3 |\sigma|^2 \le \frac 1 {8 Np}$ and $p^2\gamma^3 |\sigma\tau|^2\le \frac 1 {8 Np}$ go to $0$ uniformly over all $(\sigma,\tau)$. This proves the desired statement.
\end{proof}

We are now ready to compute expectations of $g \in  \mathcal{C}^b(\R)$ with respect to the Gibbs measure. Here the different regimes of $p$ occur. We will adopt our definition of $\tilde Z_N(1,g)$ to our requirements.
For $g \in \mathcal{C}^b(\R), g \ge 0, g \not\equiv 0$ define
\begin{equation}\label{Z'N(g)}
\tilde Z_N^1(\beta, g):= \sum_{\sigma \in \{-1, +1\}^N} g\left(\frac{|\sigma|}{N^{3/4}}\right)T(\sigma)
\end{equation}
and 
\begin{equation}\label{overlineZN(g)}
\tilde Z_N^3 (\beta, g):= \sum_{\sigma \in \{-1, +1\}^N} g\left(\frac{|\sigma|}{\sqrt{N^3p^2}}\right)T(\sigma).
\end{equation}

\begin{proposition}\label{prop:expectcrit}
At $\beta=1$ and for all $g \in \mathcal{C}^b(\R), g \ge 0, g \not\equiv 0$ we have
\begin{enumerate}[label=(\alph*)]
\item \label{expectcrit:eins} If $p^4N^3\to \infty$ we have
\begin{equation}
\lim_{N \to \infty} \frac{\E \tilde Z_N^1(\beta, g)}{2^N N^{1/4} \int_{-\infty}^{+\infty} g(x) \eee^{-\frac{1}{12} x^4}d x}= \frac1{\eee^{\frac{1}{8}} \sqrt{2\pi}}.
\end{equation}
\item \label{expectcrit:zwei} If $p^4 N^3 \to \frac {1}{c^2} \in (0, \infty)$ we have
\begin{equation}
\lim_{N \to \infty} \frac{\E \tilde Z_N^1(\beta, g)}{2^N N^{1/4}  \int_{-\infty}^{+\infty} g(x) \eee^{-c \frac{x^2}{24} -\frac{1}{12} x^4}d x}=\frac1{\eee^{\frac{1}{8}} \sqrt{2\pi}}.
\end{equation}
\item \label{expectcrit:drei} If $p^4 N^3 \to 0$ and $pN \to \infty$, then we have
\begin{equation}
\lim_{N \to \infty} \frac{\E \tilde Z_N^3(\beta, g)}{2^N Np \, \int_{-\infty}^{\infty}g(x) \eee^{-\frac 1{24} x^2} d x}=\frac1{\eee^{\frac{1}{8}} \sqrt{2\pi}}.
\end{equation}
\end{enumerate}
\end{proposition}
\begin{proof}
We start by proving Part~\ref{expectcrit:eins}.
Similar to the proof of Proposition \ref{prop:expect} we write
\begin{equation}\label{eq:expect_typical_atypical_crit}
\E \tilde Z_N^1(\beta, g)
=
\sum_{\sigma \in R_N^1}g\left(\frac{|\sigma|}{N^{3/4}}\right)\E T(\sigma)
+
\sum_{\sigma \notin R_N^1} g\left(\frac{|\sigma|}{N^{3/4}}\right)\E T(\sigma).
\end{equation}

For $\sigma\in R_N^1$ we can use the result of Lemma~\ref{ETcrit} to obtain:
\begin{eqnarray*}
\sum_{\sigma \in R_N^1}g\left(\frac{|\sigma|}{N^{3/4}}\right)\E T(\sigma)=
\sum_{\sigma \in R_N^1} g\left(\frac{|\sigma|}{N^{3/4}}\right) \exp\left( -\frac{1}{8}+p \tanh(\gamma)|\sigma|^2 + o(1)\right)
\end{eqnarray*}
with $o(1)$-term that is uniform over $R_N^1$. Note that we are in the regime where $p^4N^3\to \infty$ and
that on $R_N^1$ we have $|\sigma|^2 \le N^{3/2} h(Np)$. Thus
\begin{eqnarray*}
&&\sum_{\sigma \in R_N^1} g\left(\frac{|\sigma|}{N^{3/4}}\right) \exp\left( -\frac{1}{8}+p \tanh(\gamma)|\sigma|^2 + o(1)\right)\\
&&\quad =
\sum_{\sigma \in R_N^1} g\left(\frac{|\sigma|}{N^{3/4}}\right) \exp\left( -\frac{1}{8}+p \gamma|\sigma|^2 +
\mathcal{O}(p \gamma^3 |\sigma|^2)+ o(1)\right)\\
&&\quad =
\sum_{\sigma \in R_N^1} g\left(\frac{|\sigma|}{N^{3/4}}\right) \exp\left( -\frac{1}{8}+ \frac{1}{2N}|\sigma|^2
+ o(1)\right)
\end{eqnarray*}
by definition of $\gamma$, $h$, and $R_N^1$.

We will first show that this sum over $R_N^1$ is the dominant term, i.e.\   we claim that
\begin{equation}\label{eq:claim_untypicalcrit}
\sum_{\sigma \not\in R_N^1} g\left(\frac{|\sigma|}{N^{3/4}}\right)\E T(\sigma)
=  o\left(N^{\frac 14 }2^N \right).
\end{equation}
Letting $\|g\|_\infty:= \sup_{t\in\R} |g(t)| < \infty$ we use Lemma~\ref{ETcrit} to obtain the estimate
\begin{align*}
\left|\sum_{\sigma \not\in R_N^1} g\left(\frac{|\sigma|}{N^{3/4}}\right)\E T(\sigma)\right|
&\leq
\|g\|_\infty \sum_{\sigma \not\in R_N^1}\E T(\sigma)\\
&\leq
\|g\|_\infty \sum_{\sigma \not\in R_N^1} \eee^{-\frac{1}8+ o(1)+ p \tanh(\gamma)|\sigma|^2}\\
& = \|g\|_\infty \sum_{k: k^2 \ge N^{3/2} h(Np)}\eee^{-\frac{1}8+ o(1)+ p \tanh(\gamma)k^2}\binom{N}{\frac{N+k}2}.
\end{align*}
On a scale $2^N N^{1/4}$, the terms with $k =\pm N$ do not contribute to the above sum because $ p \tanh(\gamma)k^2 \leq  \frac 12 N$ and $\eee^{1/2}<2$.
For the other terms, we will exploit the Stirling formula
$$
\log (n!)= n \log n - n +\frac 12 \log(2 \pi) + \frac 12 \log \left(n + \frac 12\right)+\mathcal{O}(1/n),
$$
for  $n\geq 1$. The term $\log (n+\frac 12)$ differs from the usual representation with $\log n$ by $\mathcal O(1/n)$ and is included to treat the boundary cases $k=\pm n$ below.  The Stirling formula yields the crude bound
\begin{equation}\label{eq:crude_bound}
c N^{-1/2} 2^N \eee^{-N I(\frac k N) - \lambda_N(k)} \leq \binom{N}{\frac{N+k}2} \le C N^{-1/2} 2^N \eee^{-N I(\frac k N) - \lambda_N(k)},
\qquad |k|\leq N,
\end{equation}
with
\begin{equation}\label{def:I}
I(x)=\frac{1-x}2\log (1-x)+\frac{1+x}2\log (1+x) \qquad \mbox{for } x\in [-1,1]
\end{equation}
and
$$
\lambda_N(k) =  \frac 12 \log \left(\frac{(N+1)^2 - k^2}{N^2}\right).
$$
We will apply the bound in the range $\frac {N^2} 4 \leq k^2$.
Note that $I$ is an even function and that the Taylor expansion of
$NI(k/N)$ is given by
\begin{equation}\label{eq:taylor_I}
NI(k/N)= \frac{k^2}{2N}+ \sum_{j\ge 2 }d_{2j}\frac{k^{2j}}{N^{2j-1}}.
\end{equation}
The $d_i$ are the Taylor coefficients of $I(x)$, i.e $d_i=0$ if $i$ is odd and $d_i=\frac{1}{i(i-1)}$ if $i$ is even.
Moreover, note that for $N^2 \geq k^2 \ge N^2/4$ one can estimate
\begin{equation}\label{eq:logterm}
-\lambda_N(k) = -\frac 12 \log \left(\frac{(N+1)^2 - k^2}{N^2}\right)
\leq -\frac 12 \log \frac{2N+1}{N^2} \leq  -\frac 12 \log \frac{N}{N^2} \leq \frac 12 \log N.
\end{equation}

By using these facts, the inequality $\tanh(\gamma) \le \gamma$,  and aborting the Taylor expansion in \eqref{eq:taylor_I} after the fourth order term we obtain
\begin{eqnarray}\label{eq:est_k_big}
\sum_{k: N^2\geq  k^2 \ge \frac{N^2}4}\eee^{-\frac{1}8+ o(1)+ p \tanh(\gamma)k^2}\binom{N}{\frac{N+k}2}
&\le& \frac{C}{\sqrt{N}} 2^N \sum_{k: N^2 \geq  k^2 \ge \frac{N^2}4}\eee^{-\frac 1 {12} \frac{k^4}{N^3}+\frac 12 \log N}\nonumber\\
\le C 2^N N e^{-\frac{1}{3072} N}
\end{eqnarray}
because the quadratic term in the expansion of $N I(k/N)$ cancels with $p\gamma k^2$, $\eee^{\frac 12 \log N}$ cancels with $\frac 1 {\sqrt N}$, and there are at most $(2N+1)$ values for $k$.
The right hand side of \eqref{eq:est_k_big} is clearly $o(2^N N^{1/4})$.

For the  values of $k$ such that $k^2 \le \frac{N^2}4$ again by the above expansion we have
\begin{equation}\label{eq:binomial_stirling}
\binom{N}{\frac{N+k}2} \le C \frac{2^N}{\sqrt N} \eee^{-NI(k/N)}
\end{equation}
for some constant $C>0$, because now $-\lambda_N(k)$ is bounded above by a constant.
Again using the Taylor expansion \eqref{eq:taylor_I}, aborting it after the fourth order term, and using
$\tanh(\gamma) \le \gamma$ we obtain
\begin{multline*}
\sum_{k: N^{3/2} h(Np) \le k^2 \le \frac{N^2}4}\eee^{-\frac{1}8+ o(1)+ p \tanh(\gamma)k^2}\binom{N}{\frac{N+k}2}
\le \frac{C}{\sqrt{N}} 2^N \sum_{k: N^{3/2} h(Np) \le k^2 \le \frac{N^2}4}\eee^{-\frac 1 {12} \frac{k^4}{N^3}}\\
= C N^{1/4} 2^N  \frac{1}{N^{3/4}}\sum_{k: N^{3/2} h(Np) \le k^2 \le \frac{N^2}4}\eee^{-\frac 1 {12} \frac{k^4}{N^3}}.
\end{multline*}
But $\frac{1}{N^{3/4}}\sum_{k: N^{3/2} h(Np) \le k^2 \le \frac{N^2}4}\eee^{-\frac 1 {12} \frac{k^4}{N^3}}$ is a Riemann sum. We bound it from above by the Riemann integral
$$
\int_{x^2 \ge \frac 12 h(Np)} \eee^{-\frac 1{12}x^4} dx.
$$
The integral $\int_{\R} \eee^{-\frac 1{12}x^4}dx$ is finite, the domain of integration $x^2 \ge \frac 12 h(Np)$ converges to the empty set and therefore the integral goes to 0.
Thus indeed,
\begin{align}\label{eq:thus_indeed}
\left|\sum_{\sigma \not\in R_N^1} g\left(\frac{|\sigma|}{N^{3/4}}\right)\E T(\sigma)\right|
=o(N^{1/4} 2^N).
\end{align}

To estimate the sum over $\sigma \in R_N^1$, we use the following exact asymptotics also following from the Stirling formula:
\begin{equation}\label{eq:Stirling_asympt}
\binom{N}{\frac{N+k}2} =  \frac{2^N}{\sqrt{2\pi N}} \eee^{-N I(\frac k N) + o(1)},
\qquad k^2 \le N^{3/2} h(Np),
\end{equation}
where the $o(1)$-term is uniform in the specified range. We used that on $R_N^1$ the $\lambda_N(k)$-term uniformly goes to $0$.
On the other hand,
\begin{align*}
& \sum_{\sigma \in R_N^1} g\left(\frac{|\sigma|}{N^{3/4}}\right) \exp\left( -\frac{1}{8}+ \frac{1}{2N}|\sigma|^2
+ o(1)\right)\\
&=
\sum_{k: k^2 \le N^{3/2} h(Np)}g\left(\frac{k}{N^{3/4}}\right)\exp\left( -\frac{1}{8}+ \frac{1}{2N}k^2
+ o(1)\right)\binom{N}{\frac{N+k}2}\\
&= \frac{2^N}{\sqrt{2\pi N}} \sum_{k: k^2 \le N^{3/2} h(Np)}g\left(\frac{k}{N^{3/4}}\right)\eee^{-\frac{1}8+ o(1)+ \frac{1}{2N}k^2}
\eee^{-N\left(\frac{k^2}{2N}+ \sum_{j\ge 2 }d_{2j}\frac{k^{2j}}{N^{2j}}\right)+o(1)}
\end{align*}
by the above identity for the binomial coefficient and the Taylor expansion of the function $I$.

Again, the $k^2/(2N)$-term cancels. Moreover, for $k^2 \le N^{3/2} h(Np)$ only the first summands of $N\sum_{j\ge 2 }d_{2j}\frac{k^{2j}}{N^{2j}}$ survives and thus, by computing $d_4$,
$$
-N\sum_{j\ge 2 }d_{2j}\frac{k^{2j}}{N^{2j}}= -\frac 1{12}\frac{k^4}{N^3}+o(1)
$$
with a $o(1)$-term that is uniform over all $k$ with $k^2\le N^{3/2}h(Np)$.
Hence,
\begin{align*}
& \sum_{\sigma \in R_N^1} g\left(\frac{|\sigma|}{N^{3/4}}\right) \exp\left( -\frac{1}{8}+ \frac{1}{2N}|\sigma|^2
+ o(1)\right)\\
&= \frac{2^N}{\sqrt{2\pi N}} \eee^{-\frac{1}8+ o(1)}\sum_{k: k^2 \le N^{3/2} h(Np)}g\left(\frac{k}{N^{3/4}}\right)
\eee^{-\frac{1}{12}\frac{k^4}{N^3}}\\
&= \frac{2^N N^{1/4}}{\sqrt{2\pi}\, N^{3/4}} \eee^{-\frac{1}8+ o(1)}\sum_{k: k^2 \le N^{3/2} h(Np)}g\left(\frac{k}{N^{3/4}}\right)
\eee^{-\frac{1}{12}\frac{k^4}{N^3}}.
\end{align*}
Now, similar to what we did above, we have the following convergence of the Riemann sum to the Riemann integral:
$$
\lim_{N\to\infty}\frac {1}{N^{3/4}}\sum_{k\in \Z}g\left(\frac{k}{N^{3/4}}\right)
\eee^{-\frac{1}{12}\frac{k^4}{N^3}}
= \int_{\R} g(x) \eee^{-\frac{1}{12} x^4} dx.
$$
This proves \ref{expectcrit:eins}.

\medskip
The proof of \ref{expectcrit:zwei} is similar to the one of \ref{expectcrit:eins}. Again,
\begin{equation*}
\E \tilde Z_N^1(\beta, g)
=
\sum_{\sigma \in R_N^1}g\left(\frac{|\sigma|}{N^{3/4}}\right)\E T(\sigma)
+
\sum_{\sigma \notin R_N^1} g\left(\frac{|\sigma|}{N^{3/4}}\right)\E T(\sigma).
\end{equation*}
By Lemma~\ref{ETcrit},
\begin{align*}
\sum_{\sigma \in R_N^1}g\left(\frac{|\sigma|}{N^{3/4}}\right)\E T(\sigma)&=
\sum_{\sigma \in R_N^1} g\left(\frac{|\sigma|}{N^{3/4}}\right) \exp\left( -\frac{1}{8}+p \tanh(\gamma)|\sigma|^2 + o(1)\right)\\
&=
\sum_{\sigma \in R_N^1} g\left(\frac{|\sigma|}{N^{3/4}}\right) \exp\left( -\frac{1}{8}+ \frac{1}{2N}|\sigma|^2-
\frac{1}{24N^3p^2}|\sigma|^2
+ o(1)\right)
\end{align*}
with a $o(1)$-term that is uniform over $R_N^1$.

Again $\sum_{\sigma \notin R_N^1} g\left(\frac{|\sigma|}{N^{3/4}}\right)\E T(\sigma)=o(N^{1/4}2^N)$. The proof follows almost verbatim the proof of the corresponding fact in \ref{expectcrit:eins}.

On the other hand, again using identity \eqref{eq:Stirling_asympt} for the binomial coefficient, and the Taylor approximation \eqref{eq:taylor_I} of $I$:
\begin{align*}
&\sum_{\sigma \in R_N^1} g\left(\frac{|\sigma|}{N^{3/4}}\right) \exp\left( -\frac{1}{8}+ \frac{1}{2N}|\sigma|^2-
\frac{1}{24N^3p^2}|\sigma|^2
+ o(1)\right)\\
&=
\sum_{k: k^2 \le N^{3/2} h(Np)}g\left(\frac{k}{N^{3/4}}\right)\exp\left( -\frac{1}{8}+ \frac{1}{2N}k^2-\frac{1}{24N^3p^2}k^2
+ o(1)\right)\binom{N}{\frac{N+k}2}\\
&= \frac{2^N}{\sqrt {2\pi N}} \sum_{k: k^2 \le N^{3/2} h(Np)}g\left(\frac{k}{N^{3/4}}\right)\eee^{-\frac{1}8+ o(1)+ \frac{1}{2N}k^2-\frac{1}{24N^3p^2}k^2}
\eee^{-N\left(\frac{k^2}{2N^2}+ \sum_{j\ge 2 }d_{2j}\frac{k^{2j}}{N^{2j}}\right)+ o(1)}.
\end{align*}
Again, the $k^2/(2N)$-term cancels. Keeping just those terms that are not $o(1)$ when $k^2 \le N^{3/2} h(Np)$ we are left with
\begin{align*}
&\sum_{\sigma \in R_N^1} g\left(\frac{|\sigma|}{N^{3/4}}\right) \exp\left( -\frac{1}{8}+ \frac{1}{2N}|\sigma|^2-
\frac{1}{24N^3p^2}|\sigma|^2
+ o(1)\right) \\
&\quad = \frac{2^N N^{1/4}}{\sqrt{2\pi} N^{3/4}}\eee^{-\frac 18 +o(1)} \sum_{k: k^2 \le N^{3/2} h(Np)} g\left(\frac{|\sigma|}{N^{3/4}}\right)\eee^{-\frac{1}{24N^3p^2}k^2-
\frac 1{12}\frac {k^4}{N^3}}.
\end{align*}
Now, we have the assumption $N^3p^4 \to 1/c^2 \in (0, \infty)$, and therefore
$$
\lim_{N\to\infty}
\frac{1}{N^{3/4}}\sum_{k: k^2 \le N^{3/2} h(Np)} g\left(\frac{|\sigma|}{N^{3/4}}\right) \eee^{-\frac{1}{24N^3p^2}k^2-
\frac 1{12}\frac {k^4}{N^3}}
=
\int_\R e^{-\frac c{24} x^2-\frac 1 {12}x^4} g(x) dx.
$$
This proves \ref{expectcrit:zwei}.

\medskip
To show \ref{expectcrit:drei} we will use the set $R_N^2$ and decompose the sum of interest as follows
\begin{eqnarray}\label{eq:expect_typical_atypical_crit_drei}
\E \tilde Z_N^3(\beta, g)&=&\sum_{\sigma \in \{-1, +1\}^N} g\left(\frac{|\sigma|}{\sqrt{ N^3p^2}}\right)\E T(\sigma)\nonumber \\
&=&
\sum_{\sigma \in R_N^2}g\left(\frac{|\sigma|}{\sqrt{ N^3p^2}}\right)\E T(\sigma)
+
\sum_{\sigma \notin R_N^2} g\left(\frac{|\sigma|}{\sqrt{ N^3p^2}}\right)\E T(\sigma).
\end{eqnarray}
For the first summand on the right hand side notice that
\begin{align*}
&\sum_{\sigma \in R_N^2}g\left(\frac{|\sigma|}{\sqrt{ N^3p^2}}\right)\E T(\sigma)=
\sum_{\sigma \in R_N^2}g\left(\frac{|\sigma|}{\sqrt{ N^3p^2}}\right)\exp\left( -\frac{1}{8}+p \tanh(\gamma)|\sigma|^2 + o(1)\right)\\
&\quad=
\sum_{\sigma \in R_N^2}g\left(\frac{|\sigma|}{\sqrt{ N^3p^2}}\right) \exp\left( -\frac{1}{8}+ \frac{1}{2N}|\sigma|^2-
\frac{1}{24N^3p^2}|\sigma|^2
+ o(1)\right)
\end{align*}
with a $o(1)$-term that is uniform over $R_N^2$. Indeed, expanding $\tanh(\gamma)$ we see that the terms
$\frac{1}{2N}|\sigma|^2$ and $-\frac{1}{24N^3p^2}|\sigma|^2$ correspond to the first and second non-vanishing term of this expansion. All higher order terms are of the form $c_m p \gamma^m |\sigma|^2$ for some $m \ge 5$. But for $\sigma \in R_N^2$ these are bounded by $c_m pN (Np)^{2} \frac{h'(Np)}{2^m (pN)^m} \le c_m \frac{h'(Np)}{(pN)^{m-3}}$ which goes to 0 by assumption on $h'$.

For the second sum on the right hand side in \eqref{eq:expect_typical_atypical_crit_drei} we apply the following bound
\begin{align*}
&\sum_{\sigma \notin R_N^2} g\left(\frac{|\sigma|}{\sqrt{ N^3p^2}}\right)\E T(\sigma)\\
&\qquad \le
c \frac{2^N}{\sqrt N} \|g\|_\infty \sum_{k: k^2 \ge N (Np)^{2}h'(Np)}\eee^{p \tanh(\gamma)k^2-NI(k/N)-\lambda_N(k)}
\end{align*}
for some constant $c>0$
by the same considerations as before. 
We split the sum on the right hand side into two parts
\begin{multline}\label{eq:split}
\sum_{k: N^2\geq  k^2 \ge N (Np)^{2}h'(Np)}\eee^{p \tanh(\gamma)k^2-NI(k/N)-\lambda_N(k)}\\=
\sum_{k: \frac 14 N^2 > k^2 \ge N (Np)^{2}h'(Np)}\eee^{p \tanh(\gamma)k^2-NI(k/N)-\lambda_N(k)}+
\sum_{k: N^2\geq k^2 \ge \frac 14 N^2}\eee^{p \tanh(\gamma)k^2-NI(k/N)-\lambda_N(k)}.
\end{multline}
Using the very same technique as in the first part of this proof we see that
$$
\frac{2^N}{\sqrt N} \|g\|_\infty\sum_{k: N^2>k^2 \ge \frac 14 N^2}\eee^{p \tanh(\gamma)k^2-NI(k/N)-\lambda_N(k)}=o(2^N Np).
$$

Turning to the first sum on the right hand side of \eqref{eq:split} we again bound $\lambda_N(k)$ by a constant. Moreover, we expand
$\tanh(\gamma)=\sum_{j=0}^\infty c_{2j+1}\gamma^{2j+1}$ and $I$ as in \eqref{eq:taylor_I}. The quadratic terms of these expansions cancel and we are left with estimating
$$
c \frac{2^N}{\sqrt N} \|g\|_\infty \sum_{\substack{k:\\ \frac 14 N^2 \ge k^2 \ge N(Np)^{2}h'(Np)}}\eee^{p \sum_{j=1}^\infty c_{2j+1}\gamma^{2j+1} k^2-\sum_{j=2}^\infty d_{2j} k^{2j}{N^{2j-1}}}.
$$
Here the $c_j$ are the Taylor coefficients of $\tanh$, while the $d_j$ are the Taylor coefficients of $I$, as above.
Now, because the Taylor expansion of $\tanh$ is oscillating around its limit,
$$
p \sum_{j=1}^\infty c_{2j+1}\gamma^{2j+1} k^2 \le \left(-\frac 13 \gamma^3+\frac{2}{15} \gamma^5\right)pk^2\le -\frac 1{15} \gamma^3 p k^2,
$$
since $\gamma \le 1$. On the other hand
$$
-\frac 1{15} \gamma^3 p k^2\le -\frac 1{120} h'(Np)
$$
for $k$ with $k^2\ge N (Np)^{2}h'(Np)$. We have 
\begin{multline*}
\sum_{k:\atop {\frac 14 N^2 \ge k^2 \ge N (Np)^{2}h'(Np)}}\eee^{p \tanh(\gamma)k^2-NI(k/N)}
\leq
\sum_{k: \atop \frac 14 N^2 \ge  k^2 \ge N (Np)^{2}h'(Np)} \eee^{-\frac 1{120} h'(Np)-\frac{k^4}{12 N^3}}
\\\leq
C N^{3/4} \eee^{-\frac 1{120} h'(Np)}
=
o(\sqrt N (Np))
\end{multline*}
by a Riemann sum argument and our assumptions on $h'$.
Hence
$$
\frac{2^N}{\sqrt N} \|g\|_\infty \sum_{k: \frac  14 {N^2}  > k^2 \ge N (Np)^{2}h'(Np)}\eee^{p \tanh(\gamma)k^2-NI(k/N)-\lambda_N(k)}=o(Np 2^N),
$$
which proves that 
\begin{equation}\label{eq:thus_indeed2}
\sum_{\sigma \notin R_N^2} g\left(\frac{|\sigma|}{\sqrt{ N^3p^2}}\right)\E T(\sigma) = o(Np 2^N),
\end{equation}
thus estimating the second summand in~\eqref{eq:expect_typical_atypical_crit_drei}.

Concentrating on the first summand in \eqref{eq:expect_typical_atypical_crit_drei} notice that
\begin{align*}
&\sum_{\sigma \in R_N^2}g\left(\frac{|\sigma|}{\sqrt{ N^3p^2}}\right) \exp\left( -\frac{1}{8}+ \frac{1}{2N}|\sigma|^2-
\frac{1}{24N^3p^2}|\sigma|^2
+ o(1)\right)\\
&=
\sum_{k: k^2 \le N (Np)^{2}h'(Np)}g\left(\frac{k}{\sqrt{ N^3p^2}}\right) \exp\left( -\frac{1}{8}+ \frac{1}{2N}k^2-
\frac{1}{24N^3p^2}k^2
+ o(1)\right)\binom{N}{\frac{N+k}2}\\
&=\frac{2^N \eee^{-1/8+o(1)}}{\sqrt{2 \pi N}}\sum_{k:\atop k^2 \le N (Np)^{2}h'(Np)}g\left(\frac{k}{\sqrt{ N^3p^2}}\right) \eee^{\frac{1}{2N}k^2-
\frac{1}{24N^3p^2}k^2-NI(\frac k N)}\\
&=\frac{2^N \eee^{-1/8+o(1)}}{\sqrt{2 \pi N}}\sum_{k:\atop k^2 \le N (Np)^{2}h'(Np)}g\left(\frac{k}{\sqrt{ N^3p^2}}\right) \eee^{-\frac{1}{24N^3p^2}k^2}
\end{align*}
again by \eqref{eq:Stirling_asympt} and the Taylor expansion of $I$ given in \eqref{eq:taylor_I}. Note that the $o(1)$-term is uniform in $k$ with $k^2 \le N (Np)^{2}h'(Np)$ by assumption on $h'$.

Now,
$$
\lim_{N \to\infty} \frac 1{\sqrt{ N^3p^2} }\sum_{k=-N}^N g\left(\frac{k}{\sqrt{ N^3p^2}}\right) \eee^{-\frac{1}{24N^3p^2}k^2}=
\int_{-\infty}^{\infty}g(x) \eee^{-\frac 1{24} x^2} dx.
$$
This shows the assertion.
\end{proof}

As a final step, we need to control the variances of $\tilde Z_N^1(1,g)$ and $\tilde Z_N^3(1,g)$.

\begin{proposition}\label{prop:var_crit}
At $\beta=1$ and for all $g \in \mathcal{C}^b(\R), g \ge 0, g \not\equiv 0$ we have
\begin{enumerate}[label=(\alph*)]
\item \label{var1}
If $p^4N^3\to \infty$, then
$$
\V \tilde Z_N^1(\beta,g) = o((\E \tilde Z_N^1(\beta,g))^2).
$$
\item  \label{var2}
If $p^4N^3\to c \in (0, \infty)$, then
$$
\V \tilde Z_N^1(\beta,g) = o((\E \tilde Z_N^1(\beta,g))^2).
$$
\item  \label{var3}
If $p^4N^3\to 0$ and $pN\to \infty$, then
$$
\V \tilde Z_N^3(\beta,g) = o((\E \tilde Z_N^3(\beta,g))^2).
$$
\end{enumerate}
\end{proposition}
\begin{proof}
The proof combines the ideas of the proofs of Proposition \ref{prop:var} and Lemma~\ref{prop:expectcrit}.
The three cases \ref{var1}, \ref{var2}, \ref{var3} are similar. We start with a detailed proof of \ref{var1}.

\vspace*{2mm}
\noindent
\textit{Proof of \ref{var1}.}
We write:
\begin{eqnarray}\label{eq:var_typical_atypical_crit_eins}
\V \tilde Z_N^1(\beta, g)
&=&
\sum_{(\sigma,\tau) \in R_N^1}g\left(\frac{|\sigma|}{N^{3/4}}\right)g\left(\frac{|\tau|}{N^{3/4}}\right)\mathrm{Cov}(T(\sigma),T(\tau))
\nonumber \\
&& \qquad+
\sum_{(\sigma,\tau) \notin R_N^1} g\left(\frac{|\sigma|}{N^{3/4}}\right)g\left(\frac{|\tau|}{N^{3/4}}\right)\mathrm{Cov}(T(\sigma),T(\tau)).
\end{eqnarray}
We will show that the sum over the typical pairs of configurations, that is over $(\sigma,\tau)\in R_N^1$, is $o((\E \tilde Z_N^1(\beta,g))^2)$. Indeed, by Lemmas \ref{ETcrit} and \ref{ETsTtcrit} and the fact that $g$ is non-negative we see that  $\mathrm{Cov}(T(\sigma)T(\tau))=o(\E(T(\sigma)\E(T(\tau))$ uniformly over $(\sigma,\tau)\in R_N^1$,  so that, just as in the proof of Proposition \ref{prop:var},
\begin{equation}\label{eq:critvar1}
\sum_{(\sigma,\tau) \in R_N^1}g\left(\frac{|\sigma|}{N^{3/4}}\right)g\left(\frac{|\tau|}{N^{3/4}}\right)\mathrm{Cov}(T(\sigma),T(\tau))=o((\E \tilde Z_N^1(\beta,g))^2).
\end{equation}

Let us now turn to the non-typical pairs, that is to summands with $(\sigma,\tau) \notin R_N^1$. Here we cannot simply apply the local limit theorem as in Section~\ref{sec:3}. Again, denote by $V_N(k,l,m)$ the set of pairs
$$(\sigma,\tau)\in \{-1,+1\}^N\times\{-1,+1\}^N \quad \mbox{for which }|\sigma|=k,|\tau|= l, \mbox{and }|\sigma \tau| = m$$
and set $\nu_N(k,l,m) := \# V_N(k,l,m)$.

We are thus left with analyzing spin configurations that are related to the set
\begin{multline*}
\overline {\mathfrak{R}_N^1}:=\{(k,l,m)\in (\Z\cap [-N,N])^3: \\
k^2 > N^{3/2} h(Np)\mbox{ or } l^2 > N^{3/2} h(Np)\mbox{ or } m^2 > N (Np)^{1/5}\}.
\end{multline*}
Then,
\begin{eqnarray*}
&&\left|\sum_{(\sigma,\tau) \notin R_N^1} g\left(\frac{|\sigma|}{N^{3/4}}\right)g\left(\frac{|\tau|}{N^{3/4}}\right)\mathrm{Cov}(T(\sigma),T(\tau))\right|\\&&
\le \|g\|_\infty^2 \sum_{(k,l,m) \in \overline{\mathfrak{R}_N^1}}\sum_{(\sigma,\tau) \in V_N(k,l,m)}
(\E [T(\sigma) T(\tau)] + \E [T(\sigma)]\E[T(\tau)]).
\end{eqnarray*}

In the following, we estimate the sum of the terms $\E [T(\sigma) T(\tau)]$. The sum of the terms $\E [T(\sigma)]\E[T(\tau)]$ has been already estimated in~\eqref{eq:thus_indeed}. Our aim is to show that
$$
\sum_{(k,l,m) \in \overline{\mathfrak{R}_N^1}}\sum_{(\sigma,\tau) \in V_N(k,l,m)}
\E [T(\sigma) T(\tau)] = o(2^{2N}\sqrt N).
$$
Observe that by Lemma \ref{ETsTtcrit} we have
$$
\sum_{(\sigma,\tau) \in V_N(k,l,m)} \E [T(\sigma) T(\tau)]=
\nu_N(k,l,m)\eee^{-\frac{1}{4}+ p \tanh(\gamma)(k^2+l^2)+p(1-p)\tanh^2(\gamma)m^2+o(1)}.
$$

Dividing $\nu_N(k,l,m)$ by $2^{2N}$ turns it into a probability mass function which can be written  in terms of a  conditional probability as follows:
\begin{eqnarray*}
2^{-2N} \nu_N(k,l,m)&=& \P_{\text{unif}}(|\sigma|=k, |\tau|=l, |\sigma\tau|=m)\\
&=&
 \P_{\text{unif}}(|\sigma \tau|=m \,\Big|\, |\sigma|=k, |\tau|=l)\P(|\sigma|=k)\P( |\tau|=l)\\
&=& 2^{-N}\binom{N}{\frac{N+k}2}2^{-N}\binom{N}{\frac{N+l}2}\P(|\sigma \tau|=m \,\Big|\, |\sigma|=k, |\tau|=l).
\end{eqnarray*}
Here, $\P_{\text{unif}}$ is a probability distribution under which $(\sigma,\tau)$ is uniformly distributed on $\{\pm 1\}^{N}\times \{\pm 1\}^{N}$. Using hypergeometric distribution, we compute
\begin{equation}\label{eq:condprob}
\P(|\sigma \tau|=m \,\Big|\, |\sigma|=k, |\tau|=l)=
\frac{\binom{\frac{N+k}2}{\frac{N+k+l+m}4}\binom{\frac{N-k}2}{\frac{N+l-k-m}4}}{\binom{N}{\frac{N+l}2}}.
\end{equation}
Using either
\eqref{eq:crude_bound} together with $\eee^{-\lambda_N(k)} \le \sqrt N$ or alternatively the exponential Markov inequality we see that
for all $|k|\leq N$ and $|l|\leq N$ we have
\begin{align*}
&2^{-N}\binom{N}{\frac{N+k}2}\leq \exp\left(-NI(k/N)\right)\le \exp\left(-\frac{k^2}{2N}-\frac{k^4}{12N^3}\right),\\
&2^{-N}\binom{N}{\frac{N+l}2}\leq  \exp\left(-NI(l/N)\right) \le \exp\left(-\frac{l^2}{2N}-\frac{l^4}{12N^3}\right).
\end{align*}
Here again $I$ is given by \eqref{def:I} and we used its Taylor expansion \eqref{eq:taylor_I} (which consists of positive terms, only) up to fourth order.

We will first treat the cases where either $|k|\geq \frac N3$ or $|l|\geq \frac N3$.
Employing our well-known estimates $m^2\leq N^2$, $p \tanh{\gamma} \le \frac {1}{2N}$ and $p(1-p) \tanh^2(\gamma) N^2 \le \frac 1p$ we see that
\begin{eqnarray}\label{eq:outer_region1}
&&\sum_{(k,l,m)\in \overline {\mathfrak{R}_N^1}: \atop k^2 \ge \frac {N^2}{9} \vee l^2 \ge \frac{N^2}{9}} \sum_{(\sigma,\tau) \in V_N(k,l,m)} \E [T(\sigma) T(\tau)]
\nonumber\\
&\le&
\sum_{(k,l): \atop N^2\geq k^2 \ge \frac {N^2}{9} \vee N^2\geq l^2 \ge \frac{N^2}{9}} 2^{2N} \exp\left(-N(I(k/N)+I(l/N))\right)\eee^{-\frac{1}{4}+ p \tanh(\gamma)(k^2+l^2)+p(1-p)\tanh^2(\gamma)N^2+o(1)}
\nonumber\\
&\le & C 2^{2N} N^2 \exp\left(-\frac{1}{10^4}N+\frac 1p\right)=o(2^{2N})
\end{eqnarray}
for all sequences $p= p(N)$ with $pN \to \infty$. Thus
$$
\sum_{(k,l,m) \in \overline {\mathfrak{R}_N^1}: \atop k^2 \ge \frac {N^2}{9} \vee l^2 \ge \frac{N^2}{9}} \sum_{(\sigma,\tau) \in V_N(k,l,m)} \E [T(\sigma) T(\tau)]
=
o(2^{2N}\sqrt N).
$$

We now proceed to the case where $|k|< \frac N3$ and  $|l|< \frac N3$.
Recall that~\eqref{eq:crude_bound} states that
$$
c M^{-1/2} 2^M \eee^{-M I(\frac n M) - \lambda_M(n)}\le \binom{M}{\frac{M+n}2} \le C M^{-1/2} 2^M \eee^{-M I(\frac n M) - \lambda_M(n)}
$$
(for $n \in \Z, M \in \N$ and $|n| \leq   M$).  Applying  this to  the binomial coefficients in \eqref{eq:condprob} and bounding the $\log$-correction in the exponent of the denominator by $0$ we obtain that 
\begin{multline*}
\P(|\sigma \tau|=m \,\Big|\, |\sigma|=k, |\tau|=l)\le C \sqrt{\frac{N}{(N+k)(N-k)}} \\
\times \eee^{-N\left(\frac{N+k}{2N}I(\frac{N+k+l+m}{2(N+k)})+\frac{N-k}{2N}I(\frac{N-k+l-m}{2(N-k)})-I(\frac{N+l}{2N})\right)-
\lambda_{N+k}(m+l) - \lambda_{N-k}(l-m)}
\end{multline*}
for some positive constant $C$, where the cases when $k^2=N^2$ or $l^2=N^2$ have been already excluded.

Since $|k|,|l| <  N$, the point $(N+l)/(2N)$ is a convex combination of $\frac{N+k+l+m}{2(N+k)}$ and $\frac{N-k+l-m}{2(N-k)}$ with weights $\frac{N+k}{2N}$ and $\frac{N-k}{2N}$, respectively. Moreover, not only $I$ itself is a convex function, but considering its Taylor expansion (as in \eqref{eq:taylor_I})
$I(x)= \sum_{j\ge 1 }d_{2j}x^{2j}$ with positive coefficients $d_{2j}$ we see that it is a positive linear combination of convex functions. Using that $d_2=\frac 12$ we obtain
\begin{eqnarray*}
&&-N\left(\frac{N+k}{2N}I\left(\frac{N+k+l+m}{2(N+k)}\right)+\frac{N-k}{2N}I\left(\frac{N-k+l-m}{2(N-k)}\right)-I\left(\frac{N+l}{2N}\right)\right)\\
&=&-N \left(\frac{N+k}{2N}\sum_{j=1}^\infty d_{2j} \left(\frac{N+k+l+m}{2(N+k)}\right)^{2j}+
\frac{N-k}{2N}\sum_{j=1}^\infty d_{2j}\left(\frac{N-k+l-m}{2(N-k)}\right)^{2j}\right.\\
&&\qquad \left.-\sum_{j=1}^\infty d_{2j} \left(\frac{N+l}{2N}\right)^{2j}\right)\\
&\le& -\frac N2 \left(\frac{N+k}{2N}\left(\frac{N+k+l+m}{2(N+k)}\right)^2+
\frac{N-k}{2N}\left(\frac{N-k+l-m}{2(N-k)}\right)^2-\left(\frac{N+l}{2N}\right)^2\right)\\
&=& -\frac{(m-\frac{kl}N)^2}{8(N-\frac{k^2}N)},
\end{eqnarray*}
where for the inequality we used that for each $j\geq 2$
$$
\frac{N+k}{2N}d_{2j} \left(\frac{N+k+l+m}{2(N+k)}\right)^{2j}+
\frac{N-k}{2N}d_{2j}\left(\frac{N-k+l-m}{2(N-k)}\right)^{2j}-d_{2j} \left(\frac{N+l}{2N}\right)^{2j}\ge 0.
$$
This amounts to saying that
\begin{multline*}
\P(|\sigma \tau |=m \,\Big|\, |\sigma|=k, |\tau|=l) \le C \sqrt{\frac{N}{(N+k)(N-k)}} \\
\times
\exp\left(-\frac{(m-\frac{kl}N)^2}{8(N-\frac{k^2}N)}
-\lambda_{N+k}(m+l)- \lambda_{N-k}(l-m)
\right).
\end{multline*}
Keeping in mind that we already eliminated the cases where $k^2$ or $l^2$ are larger than $N^2/9$ we only have to prove that
$$
2^{-2N} \sum_{(k,l,m) \in \overline{\mathfrak{R}_N^1}\atop k^2 \le \frac{N^2}{ 9}\wedge l^2 \le \frac{N^2}{ 9}}\sum_{(\sigma,\tau) \in V_N(k,l,m)}
\E [T(\sigma) T(\tau)] = o(\sqrt N).
$$
In view of our last estimates this term can be bounded as follows:
\begin{eqnarray}
&& 2^{-2N} \sum_{(k,l,m) \in \overline{\mathfrak{R}_N^1}\atop k^2 \le \frac{N^2}{ 9}\wedge l^2 \le \frac{N^2}{ 9}}\sum_{(\sigma,\tau) \in V_N(k,l,m)}
\E [T(\sigma) T(\tau)]\nonumber \\
&\le & C \frac 1 N \sum_{(k,l,m) \in \overline{\mathfrak{R}_N^1}\atop k^2 \le \frac{N^2}{ 9}\wedge l^2 \le \frac{N^2}{ 9}} \sqrt{\frac{N}{(N+k)(N-k)}} \eee^{-\frac{1}{4}+ p \tanh(\gamma)(k^2+l^2)+p(1-p)\tanh^2(\gamma)m^2+o(1)}\nonumber \\
&&\qquad \quad  \eee^{-N \left(I\left(\frac kN\right)+I\left(\frac l N\right)\right)-\frac{(m-\frac{kl}N)^2}{8(N-\frac{k^2}N)}
-
\lambda_{N+k}(m+l)- \lambda_{N-k}(l-m) - \lambda_N(k) - \lambda_N(l)
}\nonumber\\
&\le & C \sum_{(k,l,m) \in \overline{\mathfrak{R}_N^1}\atop k^2 \le \frac{N^2}{ 9}\wedge l^2 \le \frac{N^2}{ 9}} \frac 1 N\sqrt{\frac{N}{(N+k)(N-k)}}
\eee^{-\frac{k^4}{12 N^3}-\frac{l^4}{12 N^3}+p(1-p)\tanh^2(\gamma)m^2 -\frac{(m-\frac{kl}N)^2}{8(N-\frac{k^2}N)}}\nonumber\\
&&\quad \times \eee^{
-\lambda_{N+k}(m+l)- \lambda_{N-k}(l-m)},
\nonumber
\end{eqnarray}
because $\eee^{ - \lambda_N(k) - \lambda_N(l)}\le C$, when $k^2 \le \frac{N^2}{ 9}$ and $l^2 \le \frac{N^2}{ 9}$.
Our aim is to prove that the right-hand side is $o(\sqrt N)$.

We include the set of remaining triples $(k,l,m)$ in the union of two sets:
$$
\overline{\mathfrak{R}_N^1}\setminus \left\{(k,l,m): k^2 > \frac{N^2}9 \vee l^2 > \frac{N^2}{ 9}\right\}
\subseteq \overline{\mathfrak{R}_{N,1}^1} \cup \overline{\mathfrak{R}_{N,2}^1}
$$
with
{\footnotesize
\begin{align*}
\overline{\mathfrak{R}_{N,1}^1} &:=
\{(k,l,m)\in (\Z \cap [-N,N])^3: \frac{N^2}9 > k^2 > N^{3/2} h(Np)\mbox{ or } \frac{N^2}9 > l^2 > N^{3/2} h(Np)\}, \\
\overline{\mathfrak{R}_{N,2}^1} &:=
\{(k,l,m)\in (\Z \cap [-N,N])^3: k^2 \le  N^{3/2} h(Np) , l^2 \le   N^{3/2} h(Np) ,m^2 > N (Np)^{1/5}\}.
\end{align*}
}
First we consider the case when $(k,l,m) \in \overline{\mathfrak{R}_{N,1}^1}$ and $|m|\geq \frac N5$. We claim that
\begin{align}\label{eq:aussen}
\sum_{(k,l,m) \in \overline{\mathfrak{R}_{N,1}^1} \atop
|m| \ge \frac N5} & \frac 1 N\sqrt{\frac{N}{(N+k)(N-k)}}
\eee^{-\frac{k^4}{12 N^3}-\frac{l^4}{12 N^3}+p(1-p)\tanh^2(\gamma)m^2 -\frac{(m-\frac{kl}N)^2}{8(N-\frac{k^2}N)}}\nonumber\\
& \times  \eee^{
-\lambda_{N+k}(m+l)- \lambda_{N-k}(l-m)}\to 0.
\end{align}
Indeed, as in \eqref{eq:logterm} we bound
$
\eee^{-\lambda_{N+k}(m+l)- \lambda_{N-k}(l-m)}
\le \sqrt{(N+k)(N-k)}
$
such that we only need to show that
$$
 \sum_{(k,l,m) \in \overline{\mathfrak{R}_{N,1}^1}\atop |m| \ge \frac N5}
\frac {1}{\sqrt N}
\eee^{-\frac{k^4}{12 N^3}-\frac{l^4}{12 N^3}+p(1-p)\tanh^2(\gamma)m^2 -\frac{(m-\frac{kl}N)^2}{8(N-\frac{k^2}N)}}\to 0.
$$
For $|m| \ge \frac N5$ we can estimate  $\frac{(m-\frac{kl}N)^2}{8(N-\frac{k^2}N)}\ge cN $ for some constant $c>0$, and hence
\begin{align*}
\sum_{(k,l,m) \in \overline{\mathfrak{R}_{N,1}^1}\atop
|m| \ge \frac N5}&
\frac {1}{\sqrt N} \eee^{p(1-p)\tanh^2(\gamma)m^2 -\frac{(m-\frac{kl}N)^2}{8(N-\frac{k^2}N)}}\le
C N^{2+\frac 12} \eee^{-cN+p(1-p)\tanh^2(\gamma)N^2}
\end{align*}
because the number of terms is at most $N^3$  and $m^2 \le N^2$.
Now, the right-hand side goes to $0$ because  $\tanh^2(\gamma) \le \gamma^2$,  and $pN \to \infty$.

Now we consider the case when $(k,l,m) \in \overline{\mathfrak{R}_{N,1}^1}$ and $|m| \leq  \frac N5$.
Denote
$$
\overline{\mathfrak{R'}_{N,1}^1}:= \overline{\mathfrak{R}_{N,1}^1}
\cap \left\{(k,l,m)\in \Z^3 : |m| \leq  \frac N5\right\}.
$$
In this case,   we can estimate $\eee^{-\lambda_{N+k}(m+l)- \lambda_{N-k}(l-m)}
 \le C$
for another appropriate constant $C>0$. 
In order to treat the two terms involving $m$,
consider
$$
\eee^{p(1-p)\tanh^2(\gamma)m^2 -\frac{(m-\frac{kl}N)^2}{8(N-\frac{k^2}N)}}\le \eee^{p(1-p)\tanh^2(\gamma)m^2 -\frac{(m-\frac{kl}N)^2}{8N}}
=\eee^{\frac{8pN(1-p)\tanh^2(\gamma)- 1}{8N}m^2+\frac{2klm}{8N^2} -\frac{k^2l^2}{8N^3}}.
$$
Note that $pN(1-p)\tanh^2(\gamma) \to 0$. Thus, given $\vep>0$ we have $8pN(1-p)\tanh^2(\gamma)-1\le -1 +\vep$ for $N$ large enough.  Therefore,
$$
\eee^{p(1-p)\tanh^2(\gamma)m^2 -\frac{(m-\frac{kl}N)^2}{8(N-\frac{k^2}N)}}\le
\exp\left(-\frac{\left(\sqrt{1-\vep}m-\frac{1}{\sqrt{1-\vep}}\frac{kl}{N}\right)^2}{8N}+\frac {\vep}{1-\vep}\frac{k^2l^2}{8N^3}\right).
$$
Thus, if $\vep >0$ is small enough and we write $C_1:= \frac 1{12}-\frac{\vep}{16 (1-\vep)}>0$ and $C_2:=\frac{\vep}{2(1-\vep)}>0$, then
\begin{multline}\label{eq:techn1} 
\exp\left(-\frac{k^4}{12 N^3}-\frac{l^4}{12 N^3}+p(1-p)\tanh^2(\gamma)m^2 -\frac{(m-\frac{kl}N)^2}{8(N-\frac{k^2}N)}\right) \\
  \le
\exp\left(-C_1 \frac{k^4}{N^3}-C_1\frac{l^4}{N^3}-\frac{\left(\sqrt{1-\vep}m-\frac{1}{\sqrt{1-\vep}}\frac{kl}{N}\right)^2}{8N}
-\frac{C_2(k^2-l^2)^2}{8 N^3}\right). 
\end{multline}
Applying these estimates we arrive at
\begin{eqnarray}\label{eq:riemann_sum1}
&&\sum_{(k,l,m) \in \overline{\mathfrak{R'}_{N,1}^1}} \frac 1 N \sqrt{\frac{N}{(N+k)(N-k)}}
\eee^{-\frac{k^4}{12 N^3}-\frac{l^4}{12 N^3}+p(1-p)\tanh^2(\gamma)m^2 -\frac{(m-\frac{kl}N)^2}{8(N-\frac{k^2}N)}}\nonumber \\
&\le& C\sum_{k,l:k^2> N^{3/2} h(Np)\atop \mbox{{\tiny{or}} } l^2> N^{3/2} h(Np) } \frac 1 N  \exp\left(-C_1 \frac{k^4}{N^3}-C_1\frac{l^4}{N^3}\right)
\end{eqnarray}
where we used the fact that on $\overline{\mathfrak{R'}_{N,1}^1}$ one has $$\sqrt{\frac{N}{(N+k)(N-k)}}\le C_3 \frac 1{\sqrt N}$$
for some $C_3>0$ and that for
any $k,l$  we have that
$$\sum_{m\in\Z} \frac 1{\sqrt N} \eee^{\frac{\left(\sqrt{1-\vep}m-\frac{1}{\sqrt{1-\vep}}\frac{kl}{N}\right)^2}{8N}}\le C_4$$
for a constant $C_4>0$ that does not depend on $k$ and $l$.

But for $N$ sufficiently large we can bound the Riemann sum on the right-hand side of \eqref{eq:riemann_sum1} by a Riemann integral. More precisely,
\begin{equation}\label{eq:int_bound_1}
\frac 1 {\sqrt N}\sum_{k,l:k^2> N^{3/2} h(Np)\atop \mbox{{\tiny{or}} } l^2> N^{3/2} h(Np) } \frac 1 N  \exp\left(-C_1 \frac{k^4}{N^3}-C_1\frac{l^4}{N^3}\right)\le \int_{E^1_N} \eee^{-C_1 (x^4+y^4)} d x d y .
\end{equation}
where $E^1_N:=\{ (x,y) \in \R^2: x^2 \ge   \frac 12 h(Np)  \text{ or } y^2 \geq  \frac 12 h(Np)  \}$. Note that $E^1_N \downarrow \emptyset$ as $N \to\infty$, such that the right hand side in \eqref{eq:int_bound_1} goes to $0$. This shows that
$$
\sum_{(k,l,m) \in \overline{\mathfrak{R}_{N,1}^1}}\sum_{(\sigma,\tau) \in V_N(k,l,m)}
\E [T(\sigma) T(\tau)]=
o(2^{2N} \sqrt N).
$$

On the other hand, turning to $\overline{\mathfrak{R}_{N,2}^1}$ and applying the same estimates as above
\begin{eqnarray}\label{eq:int_bound_2}
&&\frac 1 {\sqrt N} \sum_{(k,l,m) \in \overline{\mathfrak{R}_{N,2}^1}} \frac 1 N \sqrt{\frac{N}{(N+k)(N-k)}}
\eee^{-\frac{k^4}{12 N^3}-\frac{l^4}{12 N^3}+p(1-p)\tanh^2(\gamma)m^2 -\frac{(m-\frac{kl}N)^2}{8(N-\frac{k^2}N)}}\nonumber \\
&\le & C\frac 1 {\sqrt N} \sum_{(k,l,m) \in \overline{\mathfrak{R}_{N,2}^1}} \frac 1 {N^{3/2}}
\eee^{-C_1 \frac{k^4}{N^3}-C_1\frac{l^4}{N^3}-\frac{\left(\sqrt{1-\vep}m-\frac{1}{\sqrt{1-\vep}}\frac{kl}{N}\right)^2}{8N}
-\frac{C_2(k^2-l^2)^2}{8 N^3}}\nonumber\\
&&\quad\le C \frac 1 {\sqrt N} \sum_{m:m^2 \ge N (Np)^{1/5}}
\eee^{-\frac{\left(\sqrt{1-\vep}m-\frac{1}{\sqrt{1-\vep}}\sqrt N h(Np)\right)^2}{8N}}
\label{eq:int_bound_2_1}
\end{eqnarray}
for some constant $C>0$ that may change from line to line. We have used the bound $\sqrt{\frac{N}{(N+k)(N-k)}} \le \frac C {\sqrt N}$ and the fact that (assuming for concreteness that $m>0$)
$$
\frac{1}{\sqrt{1-\vep}}\frac{kl}{N} \leq \frac{1}{\sqrt{1-\vep}}\sqrt N h(Np) \leq \sqrt{1-\vep} m \text{ because } (Np)^{\frac 1{10}}/h(Np)\to\infty.
$$
Finally, we observed that
$$
\sum_{(k,l)\in\Z^2} \frac 1 {N^{3/2}}
\eee^{-C_1 \frac{k^4}{N^3}-C_1\frac{l^4}{N^3}-\frac{C_2(k^2-l^2)^2}{8 N^3}}
$$
can be bounded by a finite integral and thus itself is bounded.

Hence we may bound the right hand side of \eqref{eq:int_bound_2_1}  by a constant times
$$\int_{E_N} \eee^{-\frac {x^2}{8}}d x,$$
where
$E_N:=\{ x \in \R: x^2 \ge   (\sqrt{1-\vep}(Np)^{1/10}-\frac{1}{\sqrt{1-\vep}}h(Np))^2\}.$
By the assumption on $h(\cdot)$ we see that $E_N \downarrow \emptyset $. This, together with
$$\int_{-\infty}^{+\infty} \eee^{-\frac {x^2}{8}}d x < \infty$$
implies
$$
\sum_{(k,l,m) \in \overline{\mathfrak{R}_{N,2}^1}}\sum_{(\sigma,\tau) \in V_N(k,l,m)}
\E [T(\sigma) T(\tau)]= o(2^{2N}\sqrt N).
$$
Hence,
$$
\|g\|_\infty^2 \sum_{(k,l,m) \in \overline{\mathfrak{R}_N^1}}\sum_{(\sigma,\tau) \in V_N(k,l,m)}
\E [T(\sigma) T(\tau)]= o((\E \tilde Z_N^1(\beta,g))^2).
$$
As remarked above,
$$
\|g\|_\infty^2 \sum_{(k,l,m) \in \overline{\mathfrak{R}_N^1}}\sum_{(\sigma,\tau) \in V_N(k,l,m)}
\E T(\sigma) \E T(\tau)= o((\E \tilde Z_N^1(\beta,g))^2).
$$
Combining this with \eqref{eq:critvar1} completes the proof of \ref{var1}.

\vspace*{2mm}
\noindent
\textit{Proof of \ref{var2}.}
The proof of \ref{var2} is basically the same as the one for \ref{var1}. Indeed, again we write
\begin{eqnarray}\label{eq:var_typical_atypical_crit_zwei}
\V \tilde Z_N^1(\beta, g)
&=&
\sum_{(\sigma,\tau) \in R_N^1}g\left(\frac{|\sigma|}{N^{3/4}}\right)g\left(\frac{|\sigma|}{N^{3/4}}\right)\mathrm{Cov}(T(\sigma),T(\tau))
\nonumber \\
&& \qquad+
\sum_{(\sigma,\tau) \notin R_N^1} g\left(\frac{|\sigma|}{N^{3/4}}\right)g\left(\frac{|\sigma|}{N^{3/4}}\right)\mathrm{Cov}(T(\sigma),T(\tau)).
\end{eqnarray}
As in (a), we can use Lemmas \ref{ETcrit} and \ref{ETsTtcrit} and to see that uniformly in $(\sigma, \tau) \in~R_N^1$ we have $\mathrm{Cov}(T(\sigma)T(\tau))=o(\E(T(\sigma)\E(T(\tau))$. As in the proof of Proposition \ref{prop:var}, we have
\begin{equation}\label{eq:critvar1_(b)_bbb}
\sum_{(\sigma,\tau) \in R_N^1}g\left(\frac{|\sigma|}{N^{3/4}}\right)g\left(\frac{|\tau|}{N^{3/4}}\right)\mathrm{Cov}(T(\sigma),T(\tau))=o((\E \tilde Z_N^1(\beta,g))^2).
\end{equation}
The contribution of the pairs $(\sigma,\tau) \notin R_N^1$ can be bounded exactly as in (a).

\vspace*{2mm}
\noindent
\textit{Proof of \ref{var3}.}
The proof of \ref{var3} is not the same, but similar to the proof of~\ref{var1}. In view of Proposition~\ref{prop:expectcrit}~\ref{expectcrit:drei}, it suffices to show that
$$
\V \tilde Z_N^3(\beta, g) = o(2^{2N}(Np)^2).
$$
We again split the variance -- this time of $\tilde Z_N^3(\beta,g)$  -- into two parts:
\begin{multline}\label{eq:var_typical_atypical_crit_drei}
\V \tilde Z_N^3(\beta, g)
=
\sum_{(\sigma,\tau) \in R_N^2}g\left(\frac{|\sigma|}{\sqrt{ N^3p^2}}\right)g\left(\frac{|\tau|}{\sqrt{ N^3p^2}}\right)\mathrm{Cov}(T(\sigma),T(\tau))
\\
+
\sum_{(\sigma,\tau) \notin R_N^2} g\left(\frac{|\sigma|}{\sqrt{ N^3p^2}}\right)g\left(\frac{|\tau|}{\sqrt{ N^3p^2}}\right)\mathrm{Cov}(T(\sigma),T(\tau)).
\end{multline}
As before, by Lemmas~\ref{ETcrit}, \ref{ETsTtcrit} and the definition of $R_N^2$ we have
$$
\sum_{(\sigma,\tau) \in R_N^2}g\left(\frac{|\sigma|}{\sqrt{ N^3p^2}}\right)g\left(\frac{|\tau|}{\sqrt{ N^3p^2}}\right)\mathrm{Cov}(T(\sigma),T(\tau))=o((\E \tilde Z_N^3(\beta,g))^2),
$$
which estimates the contribution of the typical pairs.

For non-typical pairs $(\sigma,\tau) \notin R_N^2$ we proceed similarly to what we did in \ref{var1} and \ref{var2}.
First of all, in the variance of  $\tilde Z_N^3(\beta, g)$ we can disregard the cases where $k^2 := |\sigma|^2 \ge \frac {N^2}{ 9}$ or $l^2:= |\tau|^2 \ge \frac {N^2}{ 9}$ by exactly the same argument as in \eqref{eq:outer_region1}.
Defining
{\footnotesize{
$$
\overline{\mathfrak{R}_N^2} :=
\{(k,l,m)\in (\Z\cap[-N,N])^3:  k^2 > N (Np)^2 h'(Np) \vee l^2 >N (Np)^2 h'(Np) \vee m^2 > N(Np)^{1/5}\}
$$}}
we have the inclusion
$$
\overline{\mathfrak{R}_N^2}\setminus \left\{(k,l,m): k^2 \ge \frac {N^2}{ 9} \vee l^2 \ge \frac {N^2}{ 9}\right\}
\subseteq
\overline{\mathfrak{R}_{N,1}^2} \cup \overline{\mathfrak{R}_{N,2}^2}
$$
with
{\footnotesize{
\begin{align*}
\overline{\mathfrak{R}_{N,1}^2}
&:=
\{(k,l,m)\in (\Z\cap[-N,N])^3: \frac{N^2}9 > k^2 >N (Np)^2 h'(Np) \vee\frac{N^2}9 >l^2 > N (Np)^2 h'(Np)\},\\
\overline{\mathfrak{R}_{N,2}^2}
&:=
\{(k,l,m)\in (\Z\cap[-N,N])^3: k^2 \le N (Np)^2 h'(Np) , l^2 \le  N (Np)^2 h'(Np) ,m^2 > N (Np)^{1/5}\}.
\end{align*}
}}
Proceeding as in \ref{var1}, we estimate the sum over the triples $(k,l,m)\in \overline{\mathfrak{R}_N^2}$ with $k^2 \le \frac{N^2} 9$ and  $l^2 \leq \frac{N^2} 9$ as follows:
\begin{eqnarray*}
&& 2^{-2N} \sum_{(k,l,m) \in \overline{\mathfrak{R}_N^2}\atop k^2 \le \frac{N^2} 9 \wedge l^2 \leq \frac{N^2} 9 }\sum_{(\sigma,\tau) \in V_N(k,l,m)}
\E [T(\sigma) T(\tau)]\nonumber \\
&\le & C \frac 1 N \sum_{(k,l,m) \in \overline{\mathfrak{R}_N^2}\atop k^2 \le \frac{N^2} 9 \wedge l^2 \leq  \frac{N^2} 9 } \sqrt{\frac{N}{(N+k)(N-k)}} \eee^{-\frac{1}{4}+ p \tanh(\gamma)(k^2+l^2)+p(1-p)\tanh^2(\gamma)m^2+o(1)}\\
&&\qquad \quad  \times \eee^{-N (I(\frac kN)+I(\frac l N))-\frac{(m-\frac{kl}N)^2}{8(N-\frac{k^2}N)}-\lambda_{N+k}(m+l)- \lambda_{N-k}(l-m)}\\
&\le & C \sum_{(k,l,m) \in \overline{\mathfrak{R}_N^2}\atop k^2 \le \frac{N^2} 9 \wedge l^2\leq \frac{N^2} 9} \frac 1 N\sqrt{\frac{N}{(N+k)(N-k)}}
\eee^{-\frac{k^4}{12 N^3}-\frac{l^4}{12 N^3}-\frac{\delta}{N^3p^2}(k^2+l^2)+p(1-p)\tanh^2(\gamma)m^2 -\frac{(m-\frac{kl}N)^2}{8(N-\frac{k^2}N)}}\\
&&\qquad \quad \times \eee^{
-\lambda_{N+k}(m+l)- \lambda_{N-k}(l-m)},
\end{eqnarray*}
where we used the inequality $p \tanh \gamma  \leq p\gamma - 8\delta p\gamma^3 = \frac 1 {2N} - \frac{\delta}{N^3 p^2}$ for some sufficiently small $\delta>0$. Arguing as in~\eqref{eq:techn1}, we arrive at
\begin{eqnarray*}
&& 2^{-2N} \sum_{(k,l,m) \in \overline{\mathfrak{R}_N^2}\atop k^2 \le \frac{N^2} 9 \wedge l^2 \leq \frac{N^2} 9 }\sum_{(\sigma,\tau) \in V_N(k,l,m)}
\E [T(\sigma) T(\tau)]\nonumber \\
&\le&
\sum_{(k,l,m) \in \overline{\mathfrak{R}_N^2}\atop k^2 \le \frac{N^2} 9 \wedge l^2\leq \frac{N^2} 9} \frac 1 N  \sqrt{\frac{N}{(N+k)(N-k)}} \eee^{-C_1 \frac{k^4}{N^3}-C_1\frac{l^4}{N^3}-\frac{\delta}{N^3p^2}(k^2+l^2)-\frac{\left(\sqrt{1-\vep}m-\frac{1}{\sqrt{1-\vep}}\frac{kl}{N}\right)^2}{8N}
-\frac{C_2(k^2-l^2)^2}{8 N^3}}\\
&&\qquad \quad\times \eee^{
-\lambda_{N+k}(m+l)- \lambda_{N-k}(l-m)}.
\end{eqnarray*}
Now we consider the terms with $(k,l,m)\in \overline{\mathfrak{R}_{N,1}^2}$  such that additionally $|m|\geq \frac N5$. The contribution of these terms, namely
\begin{eqnarray*}
&&\sum_{(k,l,m) \in \overline{\mathfrak{R}_{N,1}^2}\atop m^2 \ge \frac {N^2}{25} } \frac 1 N  \sqrt{\frac{N}{(N+k)(N-k)}} \eee^{-C_1 \frac{k^4}{N^3}-C_1\frac{l^4}{N^3}-\frac{\delta}{N^3p^2}(k^2+l^2)-\frac{\left(\sqrt{1-\vep}m-\frac{1}{\sqrt{1-\vep}}\frac{kl}{N}\right)^2}{8N}
-\frac{C_2(k^2-l^2)^2}{8 N^3}}\\
&&\qquad \quad\times \eee^{
-\lambda_{N+k}(m+l)- \lambda_{N-k}(l-m)},
\end{eqnarray*}
goes to $0$, because the argument given in~\eqref{eq:aussen} did not depend on the choice of $p$ as long as $pN\to\infty$.

We denote by $ \overline{\mathfrak{R'}_{N,1}^2}$ the set of those $(k,l,m) \in  \overline{\mathfrak{R}_{N,1}^2}$, for which  $m^2 \le \frac{N^2}{25}$. For those triples, again, $\sqrt{\frac{N}{(N+k)(N-k)}} \le \frac{C}{\sqrt N}$ as well as
$
\eee^{-\lambda_{N+k}(m+l)- \lambda_{N-k}(l-m)} \le C
$
for some constant $C$. Moreover, 
$$
\frac 1 {\sqrt N} \sum_{|m|\leq \frac N5} \eee^{-\frac{\left(\sqrt{1-\vep}m-\frac{1}{\sqrt{1-\vep}}\frac{kl}{N}\right)^2}{8N} } \le C
$$
for yet another constant $C>0$.

Thus
\begin{eqnarray}\label{eq:int_bound_100}
&&\sum_{(k,l,m) \in \overline{\mathfrak{R'}_{N,1}^2}} \frac 1 N  \sqrt{\frac{N}{(N+k)(N-k)}} \eee^{-C_1 \frac{k^4}{N^3}-C_1\frac{l^4}{N^3}-\frac{\delta}{N^3p^2}(k^2+l^2)-\frac{\left(\sqrt{1-\vep}m-\frac{1}{\sqrt{1-\vep}}\frac{kl}{N}\right)^2}{8N}
-\frac{C_2(k^2-l^2)^2}{8 N^3}}\nonumber\\
&&\qquad \quad\times\eee^{
-\lambda_{N+k}(m+l)- \lambda_{N-k}(l-m)}\\
\nonumber &\le&
C \sum_{k,l:k^2> N (Np)^2 h'(Np) \atop \mbox{{\tiny{or}} } l^2> N (Np)^2 h'(Np) } \frac 1 N  \eee^{-\frac{\delta}{N^3p^2}(k^2+l^2)}.
\end{eqnarray}
However, for an appropriate $C>0$
$$
\frac 1 {(Np)^2} \sum_{k,l:k^2> N (Np)^2 h'(Np) \atop \text{or } l^2> N (Np)^2 h'(Np) } \frac 1 N  \eee^{-\frac{\delta}{N^3p^2}(k^2+l^2)}
\le C \int_{E_N^3} \eee^{-\delta (x^2+y^2)}d x d y.
$$
Here
$
E_N^3:=\{ (x,y)\in \R^2: x > \sqrt{h'(Np)} \mbox{ or  } y > \sqrt{h'(Np)}\}.
$
By definition of $h'$ we see that $E^3_N \downarrow \emptyset$ as $N \to\infty$, such that the right-hand side in \eqref{eq:int_bound_100} goes to $0$.
In view of the third part of Proposition \ref{prop:expectcrit} this shows that
$$
\sum_{(k,l,m) \in \overline{\mathfrak{R}_{N,2}^2}}\sum_{(\sigma,\tau) \in V_N(k,l,m)}
\E [T(\sigma) T(\tau)]= o((\E \tilde Z_N^3(\beta,g))^2).
$$

It remains to estimate the contribution of the terms coming from $\overline{\mathfrak{R}_{N,2}^2}$. Similar to what we did in the proof of \ref{var1} we  see that
\begin{eqnarray}\label{eq:int_bound_101}
&&\frac 1 {(Np)^2} \sum_{(k,l,m) \in \overline{\mathfrak{R}_{N,2}^2}} \frac 1 N  \sqrt{\frac{N}{(N+k)(N-k)}} \eee^{-C_1 \frac{k^4}{N^3}-C_1\frac{l^4}{N^3}-\frac{\delta}{N^3p^2}(k^2+l^2)-\frac{\left(\sqrt{1-\vep}m-\frac{1}{\sqrt{1-\vep}}\frac{kl}{N}\right)^2}{8N}
-\frac{C_2(k^2-l^2)^2}{8 N^3}}\nonumber\\
&&\qquad \quad\times\eee^{
-\lambda_{N+k}(m+l)- \lambda_{N-k}(l-m)}\nonumber\\
\nonumber\\
&\le&
\frac C {(Np)^2} \sum_{(k,l,m) \in \overline{\mathfrak{R}_{N,2}^2}} \frac 1 N  \sqrt{\frac{1}{N}} \eee^{-\frac{\delta}{N^3p^2}(k^2+l^2)-\frac{\left(\sqrt{1-\vep}m-\frac{1}{\sqrt{1-\vep}}\frac{kl}{N}\right)^2}{8N}}
\nonumber\\
&\le& C \sum_{m: m^2 \ge N (Np)^{\frac 15}} \sqrt{\frac{1}{N}} \eee^{-\frac{\left(\sqrt{1-\vep}m-\frac{1}{\sqrt{1-\vep}}(Np)^2h'(Np)\right)^2}{8N}}
\nonumber \\
& \le &  C \int_{E_N^4} \exp(-x^2/8) d x
\end{eqnarray}
for some constant $C>0$, by ``integrating out'' the $k$ and $l$ variables and bounding the remaining sum by the Riemann integral on the right side over the set
$$
E_N^4:=\left\{ x\in \R: x^2 > \left(\sqrt{1-\vep}(Np)^{1/10}-\frac{1}{\sqrt{1-\vep}}\frac{(Np)^2}{\sqrt N}h'(Np)\right)^2\right\}.
$$
By definition of $h'$ we see that $E^4_N \downarrow \emptyset$ as $N \to\infty$, such that the right hand side in \eqref{eq:int_bound_101} goes to $0$.

In view of the third part of Proposition \ref{prop:expectcrit} this shows that
$$
\sum_{(k,l,m) \in \overline{\mathfrak{R}_{N,2}^2}}\sum_{(\sigma,\tau) \in V_N(k,l,m)}
\E [T(\sigma) T(\tau)]= o((\E \tilde Z_N^3(\beta,g))^2).
$$

Taking the estimates over $\overline{\mathfrak{R}_{N,1}^2}$ and $\overline{\mathfrak{R}_{N,2}^2}$ together we see that
$$
\|g\|_\infty^2 \sum_{(k,l,m) \in \overline{\mathfrak{R}_N^2}}\sum_{(\sigma,\tau) \in V_N(k,l,m)}
\E [T(\sigma) T(\tau)]= o((\E \tilde Z_N^3(\beta,g))^2).
$$
As already seen in~\eqref{eq:thus_indeed2},
$$
\|g\|_\infty^2 \sum_{(k,l,m) \in \overline{\mathfrak{R}_N^2}}\sum_{(\sigma,\tau) \in V_N(k,l,m)}
\E T(\sigma) \E T(\tau)= o((\E \tilde Z_N^3(\beta,g))^2).
$$
Altogether this implies
$$
\V \tilde Z_N^3(\beta, g)= o((\E \tilde Z_N^3(\beta,g))^2),
$$
as desired.
\end{proof}

\begin{proof}[Proof of Theorem \ref{theo:crit_temp_clt}]
Using the above Propositions~\ref{prop:expectcrit} and~\ref{prop:var_crit} the proof of Theorem \ref{theo:crit_temp_clt} follows the same line as the proof of
Theorem \ref{theo:high_temp_clt}.
\end{proof}

\end{document}